\documentclass[11pt,reqno]{amsart}
\usepackage{amsthm,amssymb,amsmath}
\usepackage[pagewise]{lineno}
\usepackage{fix-cm}
\usepackage[T1]{fontenc}
\usepackage{hyperref}

\newtheorem{theorem}{Theorem}[section]
\newtheorem*{theorem*}{Theorem}
\newtheorem{lemma}{Lemma}[section]
\newtheorem{corollary}{Corollary}[section]
\newtheorem{proposition}{Proposition}[section]

\theoremstyle{definition}

\newtheorem{definition}{\bf Definition}
\newtheorem*{definition*}{\bf Definition}

\newtheorem{remark}{\sc Remark}
\newtheorem*{remark*}{\sc Remark}
\newtheorem*{remarks}{\sc Remarks}

\newtheorem*{example*}{\bf Example}

\newcommand{\loc}{{\rm loc}}

\newcommand{\sprt}{{\rm sprt\,}}

\expandafter\def\expandafter\normalsize\expandafter{%
    \normalsize
    \setlength\abovedisplayshortskip{8pt}
    \setlength\belowdisplayshortskip{8pt}
}

\begin{document}

\fontsize{10.4pt}{4.5mm}\selectfont

\title{Feller generators with singular drifts in the critical range}

\author{D.\,Kinzebulatov}

\address{Universit\'{e} Laval, D\'{e}partement de math\'{e}matiques et de statistique, Qu\'{e}bec, QC, Canada}

\email{damir.kinzebulatov@mat.ulaval.ca}

\thanks{
The research of D.K. is supported by NSERC Discovery grant (RGPIN-2024-04236)}

\author{Yu.\,A.\,Sem\"{e}nov}

\address{University of Toronto, Department of Mathematics, Toronto, ON, Canada}

\email{semenov.yu.a@gmail.com}



\begin{abstract}
We consider diffusion operator $-\Delta + b \cdot \nabla$ in $\mathbb R^d$, $d \geq 3$, with drift $b$ in a large class of locally unbounded vector fields that can have critical-order singularities. Covering the entire range of admissible magnitudes of singularities of $b$, we construct a strongly continuous Feller semigroup on the space of continuous functions vanishing at infinity, thus completing a number of results on well-posedness of SDEs with singular drifts. Our approach uses De Giorgi's method ran in $L^p$ for $p$ sufficiently large, hence the gain in the assumptions on singular drift.

For the critical borderline value of the magnitude of singularities of $b$, we construct a strongly continuous semigroup in a ``critical'' Orlicz space on $\mathbb R^d$ whose topology is stronger than the topology of $L^p$ for any $2 \leq p<\infty$ but is slightly weaker than that of $L^\infty$. 
\end{abstract}

\maketitle

\section{Introduction}

\textbf{1.~}The paper concerns with the following question: what are the minimal assumptions on a locally unbounded vector field $b:\mathbb R^d \rightarrow \mathbb R^d$, $d \geq 3$, such that operator $-\Delta + b \cdot \nabla$ generates a strongly continuous Feller semigroup? We deal with the drift singularities that substantially affect the behaviour of the heat kernel of $-\Delta + b\cdot \nabla$. For instance, the heat kernel can vanish or blow up  at some points in space. However, the Feller semigroup structure ensures that the corresponding strong Markov process exists and has a number of important properties that make it of practical interest (e.g.\,properties related to continuity, existence of invariant measure, solvability of a martingale problem \cite{BG,DZ}). It is almost impossible to survey the literature on Feller generators. We only mention some results related to the diffusion operators with irregular drifts, including drifts having strong growth at infinity \cite{MPW, PRS}, generators of distorted Brownian motion \cite{AKR,BGr2, BGr}, general locally unbounded drifts $b$ \cite{BC,Ki2,KS,Kr}. See also \cite{BCR, S, SW}.

\medskip

The question of what local singularities of drift $b$ are admissible has two dimensions: the order of  singularities (for example, for the model  singular drift $b(x)=\sqrt{\delta}\frac{d-2}{2}|x|^{-\alpha}x$ the order of singularities is determined by $\alpha-1>0$) and their magnitude (i.e.\,factor $\delta$ in the previous formula if $\alpha$ is chosen to be critical, which, as can be seen by rescaling the equation, is $\alpha=2$). The following is a large class of vector fields that can have critical-order singularities:

\begin{definition}
\label{def1}
A Borel measurable vector field $b:\mathbb R^d \rightarrow \mathbb R^d$ is said to be form-bounded if 
\begin{equation}
\label{fbd}
\|b\varphi\|_2^2 \leq \delta \|\nabla \varphi\|_2^2+c(\delta)\|\varphi\|_2^2 \quad \forall\,\varphi \in W^{1,2}
\end{equation}
for some constants $\delta$ and $c(\delta)$ (here and in what follows, $\|\cdot\|_p:=\|\cdot\|_{L^p}$, $W^{1,2}$ is the Sobolev space of functions with square integrable derivatives). Condition \eqref{fbd} is written as $b \in \mathbf{F}_\delta$. 
\end{definition}

The form-boundedness with form-bound $\delta<1$ is a classical condition on $|b|$: it provides coercivity of the corresponding to $-\Delta + b \cdot \nabla$ quadratic form in $L^2$.

\medskip

Constant $\delta$ measures the magnitude of singularities of the vector field $b$.
  If $\delta>4$, then there are various counterexamples to the regularity theory of $-\Delta + b \cdot \nabla$ and to the theory of the corresponding diffusion process. We explain below that the critical threshold value of $\delta$ is $4$.
The present paper concerns with the value of $\delta$ going up to (and including) $\delta=4$.

\medskip

\textbf{2.}~There is a plethora of results devoted to verifying inclusion $b \in \mathbf{F}_\delta$ \cite{A, CWW, CF,  F}. Here are some examples of sub-classes of $\mathbf{F}_\delta$ that appear in the literature on PDEs and stochastic differential equations (SDEs).
For example, class $\mathbf{F}_\delta$ contains vector fields $b$ from $[L^d+L^\infty]^d$ (with $\delta$ that can be chosen arbitrarily small), weak $L^d$ class, which includes
\begin{equation}
\label{hardy}
b(x)=\pm \sqrt{\delta}\frac{d-2}{2}|x|^{-2}x \in \mathbf{F}_\delta \quad (\text{but not in any }\mathbf{F}_{\delta'} \text{ with } \delta'<\delta),
\end{equation}
and, more generally, the scaling-invariant Morrey class 
\begin{equation*}
\|b\|_{M_{2+\varepsilon}}:=\sup_{r>0, x \in \mathbb R^d} r\biggl(\frac{1}{|B_r|}\int_{B_r(x)}|b|^{2+\varepsilon}dx \biggr)^{\frac{1}{2+\varepsilon}}<\infty
\end{equation*}
where $B_r(x)$ is the ball of radius $r$ centered at $x$, and $\varepsilon>0$ is fixed arbitrarily small, so $\delta=C\|b\|_{M_{2+\varepsilon}}$ for appropriate constant $C=C(\varepsilon)$ \cite{F}. 
Some other examples can be found, in particular, in \cite{Ki_survey, KiS_theory}.

\medskip

\medskip

\textbf{3.}~It was proved in \cite{KiS_sharp}, using De Giorgi's iterations in $L^p$, $p>\frac{2}{2-\sqrt{\delta}}$, and a compactness argument, that if $b \in \mathbf F_\delta$ with $\delta<4$, then the corresponding to $-\Delta + b \cdot\nabla$ SDE
\begin{equation}
\label{sde}
X_t= x - \sqrt{\delta} \frac{d-2}{2}\int_0^t b(X_s)ds + \sqrt{2}B_t,
\end{equation}
where $B_t$ is the $d$-dimensional Brownian motion,
has a martingale solution for every initial point $x \in \mathbb R^d$ (see Theorem \ref{thm_sharp} below). This is important in light of the following counterexample: if we take a particular form-bounded singular vector field $b(x)=\sqrt{\delta}\frac{d-2}{2}\mathbf{1}_{x \neq 0}|x|^{-2}x$ introducing strong attraction of $X_t$ to the origin, then, whenever
$$
\delta>4\bigg(\frac{d}{d-2}\bigg)^2,
$$
the corresponding SDE does not have a weak solution departing from $x=0$. 
Thus, the constraint $\delta<4$ in \cite{KiS_sharp} is sharp at least asymptotically (i.e.\,in high dimensions). 
It should also be added that if $\delta>4$, then for every initial point $x \neq 0$ the corresponding solution of \eqref{sde} (which, one can prove, still exists locally in time) arrives to the origin with positive probability.

\medskip

We explain where does the condition $p>\frac{2}{2-\sqrt{\delta}}$ come from in the end of this introduction. Let us add that it was known for some time that $-\Delta + b\cdot \nabla$, $b \in \mathbf{F}_\delta$, $\delta<4$, generates a strongly continuous semigroup in $L^p$, $p>\frac{2}{2-\sqrt{\delta}}$ \cite{KS}. Although this semigroup is an $L^\infty$ contraction and $p$ can be taken arbitrarily large, this result on its own does not provide a path to constructing strongly continuous Feller semigroup. 

\medskip

There already exist various methods for constructing Feller semigroup for $-\Delta + b \cdot \nabla$ with $
b \in \mathbf{F}_\delta$ with some small $\delta$. The first paper where such construction was carried out for $\delta<1 \wedge (\frac{2}{d-2})^2$, using Moser-type iterations, was \cite{KS}. 
\cite{Ki2} gave a different approach via fractional resolvent representations in $L^q$, $q>d-2$, to constructing the Feller generator, reaching the same condition on $\delta$ as in \cite{KS}, and also providing additional information about regularity of the Feller semigroup, cf.\,Theorem \ref{thm1}(\textit{v}).
All these results require $\delta \ll 1$. The reasons for this is that the argument  in \cite{KS} uses rather strong gradient bound on solutions of the corresponding elliptic equations, while the construction in \cite{Ki2} automatically provides such gradient bounds, so the Feller semigroup arises as a by-product of this construction (a more detailed discussion can be found in survey \cite{Ki_survey}).

\medskip

The question of what happens with operator $-\Delta + b \cdot \nabla$ and the corresponding parabolic equation  in the critical case $\delta=4$ was addressed in \cite{Ki_Orlicz}.  It turned out one still has a strongly continuous Markov semigroup but in Orlicz space with gauge function $\cosh-1$, moreover, the corresponding elliptic equation has a unique weak solution, and a variant of energy inequality holds. The local topology of this Orlicz space is stronger than the local topology of $L^p$ with any finite $p$, but is weaker than the topology of $L^\infty$ (\cite{Ki_Orlicz} dealt with the dynamics on the torus $\mathbb R^d/\mathbb Z^d$ or, rather on a compact Riemannian manifold).
The result of \cite{Ki_Orlicz} was summarized there as follows: strengthening the topology of the space,where the semigroup of $-\Delta + b \cdot \nabla $ is considered, allows to relax the assumptions on $\delta$\footnote{Retrospectively, condition $p>\frac{2}{2-\sqrt{\delta}}$ could be interpreted as saying the same thing, but, since semigroup $e^{t(\Delta - b \cdot \nabla)}$ in $L^p$ is automatically strongly continuous in all $L^q$, $p<q<\infty$ by the interpolation with the $L^\infty$ contraction estimate, the link between the strength of topology and the value of $\delta$ was somewhat less transparent in the $L^p$ setting.}. 
In the same vein, the Feller semigroup for $-\Delta + b \cdot \nabla $, which is acting in a space with an even stronger local topology (i.e.\,space $C_\infty$ of continuous functions vanishing at infinity with  the $\sup$-norm), should be defined for all values of $\delta$ going up to $4$. Below we show that this is indeed the case for all $\delta<4$.

\medskip

Our main results in this paper, stated briefly, are as follows.

\begin{theorem*} Let $b \in \mathbf{F}_\delta$. The following are true:

\smallskip

Theorem \ref{thm1}: If $\delta<4$, then the constructed in \cite{KiS_sharp} probability measures $\{\mathbb P_x\}_{x \in \mathbb R^d}$ solving the martingale problem for \eqref{sde}
in fact determine a Feller semigroup. Its generator is an appropriate realization of formal operator $-\Delta + b \cdot \nabla$ in $C_\infty$. This Feller semigroup is unique among Feller semigroups that can be constructed via an approximation of $b$ by bounded smooth vector fields that do not increase form-bound $\delta$ and constant $c(\delta)$.

\smallskip

Theorem \ref{thm2}: If $\delta \leq 4$ and $b$ satisfies some additional constraints on its behaviour outside of a large ball (e.g.\,bounded), then there is an analogous semigroup theory of $-\Delta + b \cdot \nabla$ but in the Orlicz space with gauge function $\cosh-1$ on $\mathbb R^d$.
\end{theorem*}

The proof of Theorem \ref{thm1} uses, in particular, some regularity results  for non-homogeneous elliptic equations obtained in \cite{Ki_multi}  by means of De Giorgi's method ran in $L^p$, and some convergence theorems obtained in \cite{KS}. This allows to verify conditions of the Trotter approximation theorem in $C_\infty$.

\medskip

Theorem \ref{thm2} is proved directly, by verifying Cauchy's criterion for solutions of the approximating parabolic equations. Let us add that in \cite{Ki_Orlicz} the volume of the torus enters the estimates, so simply blowing it up, in order to work on $\mathbb R^d$, is not an option. 
We address this in the present paper (Theorem \ref{thm2}) by working carefully  with appropriate weights.

\medskip

Theorem \ref{thm2} admits more or less direct extension to time-inhomogeneous form-bounded vector fields. On the other hand, the proof of Theorem \ref{thm1} so far uses in an essential manner (via Trotter's approximation theorem) the fact that we are working with elliptic equations, so it is limited to time-homogeneous $b=b(x)$.

\medskip

The literature on the regularity theory of diffusion operator $-\Delta + b \cdot \nabla$ and on the corresponding SDE also deals with larger classes of singular vector fields $b$, i.e.\,those that contain $\mathbf{F}_\delta$, such as the class of weakly form-bounded vector fields \cite{Ki_super,KiS_brownian} or (basically the largest possible scaling-invariant time-inhomogeneous) Morrey class \cite{Ki_Morrey}. However, in the cited papers it is essential that the form-bound $\delta$ is smaller than a dimension-dependent constant $\ll 1$, and it is not yet clear what is the critical value of $\delta$ for these classes of vector fields. There is also the Kato class of vector fields that contains drifts having strong hypersurface singularities, see e.g.\,\cite{BC}, but, on the other hand, the Kato class does not even contain $|b| \in L^d$ and itself is contained in the class of  weakly form-bounded vector fields.

\medskip

\textbf{4.~}As was mentioned above, if $b \in \mathbf{F}_\delta$, $\delta<4$, then one can construct a quasi contraction strongly continuous Markov semigroup $e^{-t\Lambda_p}$ in $L^p$, $\Lambda_p \supset -\Delta + b \cdot \nabla$, $p \in ]\frac{2}{2-\sqrt{\delta}},\infty[$. 
We proved in \cite{KiS_theory} that the last statement remains valid for all $p$ in a larger interval $$I_c:=[\frac{2}{2-\sqrt{\delta}},\infty[ \quad \text{(``interval of quasi contractive solvability'')},$$ moreover, the corresponding semigroup inherits many important properties of the heat semigroup $e^{t\Delta}$ such as $L^p \rightarrow L^q$ bounds and holomorphy.
The interval of quasi contractive solvability $I_c$ can be further extended to the interval of quasi bounded solvability $$I_m:=]\frac{2}{2-\frac{d-2}{d}\sqrt{\delta}},\infty[,$$ i.e.\,for all $p \in I_m$ one still has a strongly continuous semigroup $e^{-t\Lambda_p}$, $\Lambda_p \supset -\Delta + b \cdot \nabla$, but now it satisfies a weaker bound
$$
\|e^{-t\Lambda_p}\|_p \leq M_{p,\delta} e^{\lambda_{p,\delta} t}\|f\|_p \quad \text{ for some } M_{p,\delta}>1.
$$
The interval of quasi bounded solvability $I_m$ is sharp. See \cite{KiS_theory}. We note that if $\delta \uparrow 4$, then, while the interval of quasi contractive solvability $I_c$ tends to the empty set, the interval of quasi bounded solvability $I_m$ tends to a non-empty interval $]\frac{d}{2},\infty[$. That said, as $\delta \uparrow 4$, one has $M_{p,\delta} \uparrow \infty$, so this result still does not allow to include $\delta=4$.

\medskip

Where does the condition $\delta<4$, $p\in I_c$, come from can be seen from the following elementary calculation. Let $b\in\mathbf F_\delta$ be additionally bounded and smooth. Consider Cauchy problem
 $(\partial_t - \Delta +b \cdot \nabla)u=0$, $u(0)=f \in C_c^\infty$. Without loss of generality, $f \geq 0$, and so $u \geq 0$. Set $v=e^{-\lambda t}u$, $\lambda>0$. Multiply equation  $(\lambda + \partial_t - \Delta +b \cdot \nabla)v=0$ by $v^{p-1}$ and integrate by parts: 
$$
\lambda \langle v^p\rangle + \frac{1}{p}\langle \partial_t v^p\rangle + \frac{4(p-1)}{p^2}\langle |\nabla v^{\frac{p}{2}}|^2 \rangle + \frac{2}{p}\langle b \cdot \nabla v^{\frac{p}{2}},v^{\frac{p}{2}}\rangle=0,
$$
($\langle\cdot\rangle$ denotes the integration over $\mathbb R^d$, $\langle\cdot,\cdot \rangle$ is the inner product in $L^2$ over reals).

Applying quadratic inequality in the last term, we arrive at
$$
p\lambda \langle v^p\rangle + \langle \partial_t v^p\rangle + \frac{4(p-1)}{p}\langle |\nabla v^{\frac{p}{2}}|^2 \rangle \leq \alpha \langle |b|^2,v^p \rangle + \frac{1}{\alpha} \langle |\nabla v^{\frac{p}{2}}|^2 \rangle
$$
Now, applying $b \in \mathbf{F}_\delta$ and selecting $\alpha=\frac{1}{\sqrt{\delta}}$, we obtain
\begin{equation*}
\biggl[p\lambda  - \frac{c(\delta)}{\sqrt{\delta}}\biggr]\langle v^p\rangle + \langle \partial_t v^p\rangle + \biggl[\frac{4(p-1)}{p}-2\sqrt{\delta} \biggr]\langle |\nabla v^{\frac{p}{2}}|^2 \rangle \leq 0, \quad \lambda \geq \frac{c(\delta)}{p\sqrt{\delta}}.
\end{equation*}
In order to keep the dispersion term non-negative, one needs $\frac{4(p-1)}{p}-2\sqrt{\delta}\geq 0$, i.e.\,$\delta<4$ and $p \in I_c$, which then yields $\|u\|_p\leq e^{\frac{c(\delta)t}{p\sqrt{\delta}}}\|f\|_p$.

\subsection*{Notations}  $B_r(x)$ denotes the open ball of radius $r$ centered at $x \in \mathbb R^d$, $B_r:=B_r(0)$.

Let $\mathcal B(X,Y)$ denote the space of bounded linear operators between Banach spaces $X \rightarrow Y$, endowed with the operator norm $\|\cdot\|_{X \rightarrow Y}$. $\mathcal B(X):=\mathcal B(X,X)$. 

The space of $d$-dimensional vectors with entries in $X$ is denoted by $[X]^d$.

We write $T=s\mbox{-} Y \mbox{-}\lim_n T_n$ for $T$, $T_n \in \mathcal B(X,Y)$ if $$\lim_n\|Tf- T_nf\|_Y=0 \quad \text{ for every $f \in X$}.
$$ 

Put $L^p=L^p(\mathbb R^d)$, $W^{1,p}=W^{1,p}(\mathbb R^d).$
Set $\|\cdot\|_p:=\|\cdot\|_{L^p}$
and
$
\|\cdot\|_{p \rightarrow q}:=\|\cdot\|_{L^p \rightarrow L^q}.
$

Put
$$
\langle f,g\rangle = \langle f g\rangle :=\int_{\mathbb R^d}f g dx$$ 
(all functions considered in this paper are real-valued).

$C_c$ ($C_c^\infty$) denotes the space of continuous (smooth) functions on $\mathbb R^d$ having compact support.

$C_\infty:=\{f \in C(\mathbb R^d) \mid \lim_{x \rightarrow \infty}f(x)=0\}$ endowed with the $\sup$-norm.

Set
$$
\gamma(x):=\left\{
\begin{array}{ll}
c\exp\left(\frac{1}{|x|^2-1}\right)& \text{ if } |x|<1, \\
0, & \text{ if } |x| \geqslant 1,
\end{array}
\right.
$$
where $c$ is adjusted to $\int_{\mathbb R^d} \gamma(x)dx=1$, and put $\gamma_\varepsilon(x):=\frac{1}{\varepsilon^{d}}\gamma\left(\frac{x}{\varepsilon}\right)$, $\varepsilon>0$, $x\in \mathbb R^d$.
Define the Friedrichs mollifier of a function $h \in L^1_{\loc}$ (or a vector field with entries in $L^1_{\loc}$) by $$E_\varepsilon h:=\gamma_\varepsilon \ast h.$$

\bigskip

\section{Feller semigroup in regime $\delta<4$}

For a given $b \in \mathbf{F}_\delta$, define $b_n:=E_{\varepsilon_n} b$ ($\varepsilon_n \downarrow 0$), where $E_\varepsilon$ is the Friedrichs mollifier. Then $b_n$ are bounded, smooth, converge to $b$ component-wise locally in $L^2_{\loc}$, and do not increase the form-bound $\delta$ and constant $c(\delta)$ of $b$, i.e.
$$
\|b_n\varphi\|_2^2 \leq \delta \|\nabla \varphi\|_2^2+c(\delta)\|\varphi\|_2^2 \quad \forall\,\varphi \in W^{1,2}
$$
(see e.g.\,\cite{KiS_MAAN} for the proof). 
By the classical theory, for every $n \geq 1$, Cauchy problem
$$
(\partial_t+\Lambda_n)u_n=0, \quad u_n(0)=f \in C_\infty,
$$
$$
\text{where } \Lambda_n:=-\Delta + b_n \cdot \nabla, \quad D(\Lambda_n):=(1-\Delta)^{-1} C_\infty,
$$
has unique solution $u_n(t,x)=:e^{-t\Lambda_n}f(x)$, and $e^{-t\Lambda_n}$ is a strongly continuous Feller semigroup on $C_\infty$.

\medskip

Put $\rho_x(y):=\rho(y-x)$, $
\rho(y)=(1+\kappa|y|^{2})^{-\frac{d}{2}-1}$, $y \in \mathbb R^d$.

\begin{theorem}[1st Main Result]
\label{thm1}
Let $b \in \mathbf{F}_\delta$ with $\delta<4$. Then

{\rm (\textit{i})} The limit
$$
s\mbox{-}C_\infty\mbox{-}\lim_{n}e^{-t\Lambda_n} \quad (\text{loc.\,uniformly in $t \geq 0$})
$$
exists and is a strongly continuous Feller semigroup on $C_\infty$, say, $e^{-t\Lambda}$. Its generator $\Lambda$ is an appropriate operator realization of the formal operator $-\Delta + b \cdot \nabla$ in $C_\infty$ (in general, no longer an algebraic sum of $-\Delta$ and $b \cdot \nabla$, see remark after the theorem regarding domain $D(\Lambda)$).

\smallskip

{\rm (\textit{ii})}  Feller semigroup $e^{-t\Lambda}$ is unique in the sense of approximations, i.e.\,does not depend on the choice of a bounded smooth approximation $b_n$ of $b$ in (\textit{i}), as long as $b_n$ converge to $b$ in $[L^2_{\loc}]^d$ and do not increase the form-bound $\delta$ of $b$ and constant $c(\delta)$.

\smallskip

{\rm (\textit{iii})} Strong Feller property for resolvent:
\begin{align}
\tag{\cite{Ki_multi}}
\|(\mu+\Lambda)^{-1}f\|_{C_\infty}  \leq K \sup_{x \in \frac{1}{2}\mathbb Z^d}\biggl((\mu-\mu_1)^{-\frac{1}{p\theta}} \langle |f|^{p\theta}\rho_x\rangle^{\frac{1}{p\theta}} \notag  + \mu^{-\frac{\beta}{p}}\langle |f|^{p\theta'}\rho_x\rangle^{\frac{1}{p\theta'}} \biggr), \quad f \in L^{p\theta} \cap L^{p\theta'}
\end{align}
for fixed $1<\theta<\frac{d}{d-2}$ and $p \geq 2$ such that $p>\frac{2}{2-\sqrt{\delta}}$, for all $\mu$ strictly greater than certain $\mu_1$. In particular, taking into account that $\langle \rho_x \rangle<\infty$, we have, appealing to the Dominated convergence theorem,
$$
\|(\mu+\Lambda)^{-1}f\|_{C_\infty} \leq C \|f\|_\infty, \quad f \in L^\infty.
$$

\smallskip

{\rm (\textit{iv})} For all $\frac{2}{2-\sqrt{\delta}} \leq p \leq q < \infty$, 
\begin{equation}
\tag{\cite{KiS_theory,S}}
\|e^{-t\Lambda_p}\|_{p \rightarrow q} \leq C_{\delta,d}e^{\omega_{p} t}t^{-\frac{d}{2}(\frac{1}{p}-\frac{1}{q})}, \quad \omega_p=\frac{c(\delta)}{2(p-1)}.
\end{equation}

\smallskip

{\rm (\textit{v})} If additionally $\delta<\frac{4}{(d-2)^2} \wedge 1$, then the resolvent $u=(\mu+\Lambda)^{-1}f$ satisfies, for every $q \in [2,\frac{2}{\sqrt{\delta}}[$,
\begin{equation}
\label{i}
\tag{\cite{KS}}
\|\nabla u\|_q\leq K_1(\mu-\mu_0)^{-\frac{1}{2}}\|f\|_q,\quad \|\nabla |\nabla u|^{\frac{q}{2}} \|_2  \leq K_2(\mu-\mu_0)^{-\frac{1}{2}+\frac{1}{q}}\|f\|_q, 
\end{equation}
and
\begin{equation}
\label{ii}
\tag{\cite{Ki2}}
\|(\mu-\Delta)^{\frac{1}{2}+\frac{1}{s}}u\|_q \leq K\|(\mu-\Delta)^{-\frac{1}{2}+\frac{1}{r}}f\|_q, \quad \text{ for all } 2 \leq r<q<s
\end{equation}
for all $\mu$ greater than some generic $\mu_0$. In particular, we can select $q>d-2$ (and, in the second assertion, $s$ close to $q$) so that, by the Sobolev embedding theorem, the elements on the domain $D(\Lambda)$ are H\"{o}lder continuous. 
\end{theorem}

\begin{remarks}
1.~A crucial feature of assertions (\textit{i})-(\textit{iii}) is that they cover the entire range $0<\delta<4$ of magnitudes of singularities of $b$.

\medskip

2.~Assertions (\textit{iv}), (\textit{v}) are included for the sake of completeness. Assertion (\textit{v}) demonstrates that as  $\delta$ becomes smaller the information that we have about the Feller generator $\Lambda$ becomes more detailed. 

\medskip

3.~The Feller semigroup $e^{-t\Lambda}$ from Theorem \ref{thm1} determines probability measures $\{\mathbb P_x\}_{x \in \mathbb R^d}$ on the canonical space of c\`{a}dl\`{a}g trajectories $\omega_t$, i.e.\,$$e^{-t\Lambda}f(x)=\mathbb E_{\mathbb P_x}f(\omega_t), \quad f \in C_\infty.$$  By a classical result, the process
$$
t \mapsto u(\omega_t)-u(x) + \int_0^t \Lambda u(\omega_s) ds, \quad u \in D(\Lambda), \quad \omega \text{ is c\`{a}dl\`{a}g}, 
$$
is a $\mathbb P_x$-martingale.
That said, there is no description of the domain $D(\Lambda)$ of generator $\Lambda$ even if $|b| \in L^\infty$ with compact support; one can be certain that $C_c^\infty \not\subset D(\Lambda)$. So, for the continuous martingale characterization of $\mathbb P_x$, we have the following results.

\begin{theorem}[\cite{Ki_multi,KiS_sharp}]
\label{thm_sharp}
Let $b \in \mathbf{F}_\delta$ with $\delta<4$. 

1) \cite{KiS_sharp} For every $x \in \mathbb R^d$ there exists a martingale solution of SDE \eqref{sde}, i.e.\,a probability measure $\mathbb P_x$ on the canonical space of continuous trajectories $\bigl(C([0,1],\mathbb R^d),\mathcal B_t=\sigma\{\omega_s \mid 0 \leq s \leq t\}\bigr)$, such that $\mathbb P_x[\omega_0=x]=1$,
$$
\mathbb E_{x}\int_0^t|b(\omega_s)|<\infty, \quad 0<t\leq 1 \qquad (\mathbb E_x:=\mathbb E_{\mathbb P_x})
$$ 
and, for every $\varphi \in C_2^2$, the process
$$
M^\varphi_t:=\varphi(\omega_t)-\varphi(\omega_0) + \int_0^t (-\Delta \varphi + b \cdot \nabla \varphi)(s,\omega_s) ds
$$
is a continuous martingale, so $$\mathbb E_x[M^\varphi_{t_1} \mid \mathcal B_{t_0}]=M_{t_0}^\varphi$$ for all $0 \leq t_0<t_1 \leq 1$ $\mathbb P_x$-a.s.

2) \cite{Ki_multi} The probability measures $\{\mathbb P_x\}_{x \in \mathbb R^d}$ are unique in the sense of approximation (Theorem \ref{thm1}(\textit{ii})) and  constitute a strong Markov family. 
\end{theorem}

The probability measures from Theorems \ref{thm1} and \ref{thm_sharp} are obtained via the same approximation of $b$ by $b_n$ and thus coincide.
Moreover, as follows from Theorem \ref{thm1}, the probability measures $\{\mathbb P_x\}_{x \in \mathbb R^d}$ from Theorem \ref{thm_sharp} determine a strongly continuous Feller semigroup on $C_\infty$ by formula $e^{-t\Lambda}f(x):=\mathbb E_{\mathbb P_x}f(x)$.

Together with the conditional weak uniqueness results of \cite{Ki_Morrey,KiM} for SDE \eqref{sde} and the strong well-posedness result of \cite{KiM_strong} (via the approach of R\"{o}ckner-Zhao), we consider Theorem \ref{thm1} as tentatively completing the description of the diffusion process with form-bounded drift $b$ for $\delta<4$.

\end{remarks}

\bigskip

\section{Semigroup in Orlicz space in the critical regime $\delta=4$}

Here we treat the borderline case $\delta=4$ which forces us to consider the problem in a suitable Orlicz space. Namely, put
\[
\Phi(t)=\cosh t-1,\quad \cosh t:=\frac{e^t+e^{-t}}{2}, \quad t\in\mathbb R.
\]
Clearly, this function is convex, $\Phi(t)=\Phi(|t|)$, $\Phi(t)/t\to 0$ as $t\to 0$, $\Phi(t)/t\to\infty$ as $t\to\infty$, and $\Phi(t)=0$ if and only if   $t=0$. So the space $\mathcal L_\Phi=\mathcal L_\Phi(\mathbb R^d)$ of real-valued $\mathcal L^d$ measurable functions on $\mathbb R^d$ endowed with the gauge norm
\[
\|f\|_\Phi=\inf\{c>0\mid\langle\cosh\frac{f}{c}-1\rangle\leq 1\}
\]
is a Banach space (recall that $\langle \cdot\rangle$ denotes integration over $\mathbb R^d$).

Note that $$\Phi(t)=\int_0^t\sinh \tau d\tau, \quad \Phi(t)=\sum_{m=1}^\infty\frac{t^{2m}}{(2m)!}\quad \text{ and } \quad\Big\langle\Phi\Big( \frac{f}{\|f\|_\Phi}\Big)\Big\rangle\leq 1.$$
 In particular, 
\begin{equation}
\label{f_ineq}
\|f\|_{2m}\leq \big((2m)!\big)^\frac{1}{2m} \|f\|_\Phi,\; m=1,2,\dots, 
\end{equation}
so
 \[
f\in\mathcal L_\Phi\Rightarrow f\in L^p \text{ and } \lim_n\|f_n-f\|_\Phi=0\Rightarrow\lim_n\|f_n-f\|_p=0
 \]
for each $ p\in[2,\infty[$ and $f_n \in\mathcal L_\Phi$. 

\begin{definition}
Let $L_\Phi$ denote the closure of $C_c^\infty$ with respect to gauge norm $\|\cdot\|_\Phi$. This is our Orlicz space.
\end{definition}

It follows from \eqref{f_ineq} that locally the topology in $L_\Phi$ is weaker than the topology in $L^\infty$. On the other hand, the functions in $L_\Phi$ must vanish at infinity sufficiently rapidly, i.e.\,in particular, no slower than functions in $L^2$. We also note that $\mathcal S\subset L_\Phi$.

\begin{theorem}[2nd Main Result]
\label{thm2}
 Assume that $b \in \mathbf{F}_4$, i.e. 
\begin{equation}
\label{b_cond}
\|b\varphi\|_2^2 \leq 4 \|\nabla \varphi\|_2^2+c(4)\|\varphi\|_2^2 \quad \forall\,\varphi \in W^{1,2}
\end{equation}
and that $b$ has compact support: $$\sprt b \subset B_{R_1} \quad \text{for some } R_1<\infty.$$
Let $\{b_n\}_{n\geq 1}$ be any sequence of $C^\infty$ smooth vector fields that satisfy \eqref{b_cond} with the same constants as $b$ and are such that $$\lim_{n\to\infty}\|b-b_n\|_2 =0\ \quad \text{ and } \cup_{n\geq n_0}\sprt b_n\subset B_R \text{ for some $R$, where $R_1<R<\infty$ and $n_0 \gg 1$}$$
(e.g.\,one can take $b_n:=E_{\varepsilon_n} b$, $\varepsilon_n \downarrow 0$, where $E_{\varepsilon}$ is the Friedrichs mollifier. Then $\sprt \;b_n\subset B_{R_1+\frac{1}{n}}$). 
Let $u_n=u_n(t,x)$ denote the classical solution to Cauchy problem
\begin{align*}
(\partial_t -\Delta + b_n\cdot \nabla)u_n=0, \quad u_n(0)= f\in C_c^\infty. 
\end{align*}
Put
$$T^t_nf:=u_n(t), \quad t \geq 0.$$
The following are true:

\smallskip

{\rm (\textit{i})} For every $n \geq 1$, 
\[
T_n^tf\in L_\Phi \text{ and } \|T_n^tf\|_\Phi\leq e^{\omega t}\|f\|_\Phi, \quad f \in C_c^\infty(\mathbb R^d),
\]
where constant $\omega \geq 0$ depends only on $d$, $c(4)$, $R$.

The operators $\{T_n^t\}_{t \geq 0}$ extend by continuity to a positivity preserving quasi contraction strongly continuous semigroup in $L_\Phi$, say, $e^{-t\Lambda_n}$.
Its generator $\Lambda_n$ in an appropriate operator realization of $-\Delta + b_n \cdot \nabla$ in $L_\Phi$.

\smallskip

{\rm (\textit{ii})} The limit
$$s\mbox{-}L_\Phi\mbox{-}\lim_n e^{-t\Lambda_n} \quad (\text{loc.\, uniformly in $t \geq 0$})$$ 
exists and determines a positivity preserving quasi contraction strongly continuous semigroup in $L_\Phi$, say, $e^{-t\Lambda}$. For every $g \in L_\Phi$, 
$
u:=e^{-t\Lambda }g
$
satisfies $u \in L^2_{\loc}([0,\infty[,W^{1,2})$ and
is a weak solution to parabolic equation $(\partial_t-\Delta + b \cdot \nabla)u=0$ in the sense that
$$
-\langle u,\partial_t \psi\rangle + \langle \nabla u, \nabla \psi\rangle + \langle b \cdot \nabla u,\psi\rangle=0 \quad \text{ for all } \psi \in C_c^1(]0,\infty[ \times \mathbb R^d).
$$

\smallskip

{\rm (\textit{iii})} The semigroup $T^t$ is unique in the sense of approximations, i.e.\,does not depend on the choice of a regularization $b_n$ of $b$ as long as $b_n$ satisfies the above assumptions.
\end{theorem}

\subsection{Some extensions of Theorem \ref{thm2}}
\label{ext_sect}

We can remove the assumption of the compact support of $b$, but we still need some assumptions on the rate of decay of $b$ at infinity. That is,
assume that vector field $b:\mathbb R^d\rightarrow\mathbb R^d$ can be represented as the sum
$b=b^{(1)}+b^{(2)}$
where
\[
b^{(1)}\in\mathbf F_4, \quad  b^{(2)} \in [L^\infty\cap L^2]^d
\]
are such that
\[
\sprt b^{(1)}\subset B_{R} \;\text{ and }\; \sprt b^{(1)}\cap \sprt b^{(2)}=\varnothing.
\]
Set  $b_n^{(1)}=\mathbf 1_{|b^{(1)}|\leq n}b^{(1)}$ and put
$
b_n:= b_n^{(1)}+b^{(2)}.
$

\begin{theorem}
\label{thm3}
Let $b$ be as above. Then the assertions of Theorem \ref{thm2} remain valid with the following modification: for every $n \geq 1$,
\[
\|T_n^t f\|_\Phi\leq e^{(\lambda+G)t} \|f\|_\Phi, \quad t \geq 0,
\]
where $\lambda= 2^{-1}c_5+2^{-1}\|b^{(2)}\|_\infty^2$ and $G= c_5\langle\mathbf 1_{B_{aR_1}}\rangle+\|b^{(2)}\|_2^2$, $c_5=c(4)+4(d-1)R^{-2}_1
$, $a=1+\theta^{-1}$ (constants $c_5$, $a$ and $\theta$ are from the proof of Theorem \ref{thm2}).
\end{theorem}

The previous theorem applies to vector field
$$b(x)=(d-2)\mathbf 1_{B_R}(x)|x|^{-2}x + C\mathbf 1_{B^c_R}(x)|x|^{-\alpha-1}x,\quad \alpha>\frac{d}{2},$$ 
where $R>0$, $C<\infty$. 
That said, a model example of a vector field $b \in \mathbf{F}_\delta$ having critical-order singularity at the origin and critical decay at infinity is
\begin{equation}
\label{model}
b(x)=\frac{\sqrt{\delta}}{2}(d-2)|x|^{-x}x
\end{equation}
(note the sign in front of $\sqrt{\delta}$).
As the previous example shows, Theorem \ref{thm3} allows us to take $\delta=4$, but it still imposes a stronger requirement, in comparison with \eqref{model}, on the rate of decay of $b$ outside of a ball of large radius. The next theorem and the remark after address that.

\begin{theorem}
\label{thm4}
 Let $|b|\in L^2$, ${\rm div}\, b\in L^1_\loc$. Set $V=0\vee {\rm div}\,b$ and assume that $V=V_1+V_2$,
 $$V_2 \in L^\infty, \quad \|V_1^\frac{1}{2}\varphi\|_2^2\leq 4\|\nabla\varphi\|_2^2+c(4)\|\varphi\|_2^2 \text{ for all } \varphi\in W^{1,2},\text{ and } \sprt V_1 \subset B_{R_1}.$$
 Then the assertions of Theorem \ref{thm2} remain valid with the following modification: for every $n \geq 1$,
 \[
 \|T_n^tf\|_\Phi\leq e^{(\lambda +\|V_2\|_\infty+G)t}\|f\|_\Phi, \quad t \geq 0,
 \]
where $\lambda=c(4)+4(d-1)R^{-2}$, $G=2\lambda\langle\mathbf 1_{B_{aR_1}}\rangle$, $a=3$.
\end{theorem}

Furthermore, one can remove condition $|b|\in L^2$ in Theorem \ref{thm4} by considering $\tilde{b}=b+\mathsf{f}$, where $b$ satisfies the assumptions of Theorem \ref{thm4}, and $|\mathsf{f}|\in L^\infty$, ${\rm div\,} \mathsf{f}\in L^1_\loc$, $V_3:=0\vee {\rm div\,} \mathsf{f}\in L^\infty$. See Remark \ref{thm4_rem} after the proof of Theorem \ref{thm4} for details.
This allows to include model drift \eqref{model}, i.e.\,take
$$\tilde{b}(x)=(d-2)|x|^{-2}x.$$ (Set $\tilde{b}_n=(d-2)E_n(\mathbf 1_{B_1}|x|^{-2}x)+(d-2)\mathbf 1_{B_1^c}|x|^{-2}x$.)

\begin{remark}
One can combine drifts considered in the previous theorems, e.g.\,one can consider drift $b+\mathsf{f}$ with $b$ from Theorem \ref{thm3} and $\mathsf{f}$ from Theorem \ref{thm4}, such that $$b^{(1)}\in \mathbf F_{\delta_1}, \quad V_1^\frac{1}{2}\in \mathbf F_{\delta_2}, \quad \delta_1+\delta_2=4.$$
\end{remark}

The main disadvantage of the previous results is that the singularities of $b$ are contained in a bounded set. In the next theorem we improve these results as follows.

\begin{theorem}
\label{thm5}
Let $\{x_m\}\subset\mathbb R^d$, $\{R_m\}\subset\mathbb R_+$ be such that $\lim_m|x_m|=\infty$ and $B(x_m,R_m)\cap B(x_k,R_k)=\emptyset$ for all $m\neq k$. Let $b(x)=\sum_{m=1}^\infty b^{(m)}(x)$ be such that
\[
\sprt \;b^{(m)}\subset B(x_m,R_m), \; \|b^{(m)}\varphi\|_2^2\leq \delta_m\|\nabla\varphi\|_2^2+c(\delta_m)\|\varphi\|_2^2 \quad \varphi\in W^{1,2},
\]
\[
\sum_{m=1}^\infty\delta_m=4, \;\;\sum_{m=m_0}^\infty(R_m^{-2}+R_m^d)\delta_m<\infty, \text{ and } \sum_{m=m_0}^\infty (1+R_m^d)c(\delta_m)<\infty \text{ for some } m_0>>1.
\]
Then all assertions of Theorem \ref{thm2} remain valid.
\end{theorem}

\bigskip

\section{Proof of Theorem \ref{thm1}}

Assertion (\textit{i}) will follow from the Trotter approximation theorem, which, applied to semigroups $\{e^{-t\Lambda_n}\}_{n \geq 1}$ in $C_\infty$, can be formulated as follows:

\begin{theorem}[see {\cite[IX.2.5]{K}}]
Assume that exists $\mu_0>0$ independent of $n$ such that

\smallskip

$1$) $\sup_n\|(\mu+\Lambda_n)^{-1}f\|_\infty \leq \mu^{-1}\|f\|_\infty$, $\mu \geq \mu_0$;

\smallskip

$2$) there exists $\text{\small $s\text{-}C_\infty\text{-}$}\lim_{n} (\mu+\Lambda_n)^{-1}$ for some $\mu \geq \mu_0$;

\smallskip

$3$) $\mu (\mu+\Lambda_n)^{-1} \rightarrow 1$ in $C_\infty$ as $\mu \uparrow \infty$ uniformly in $n$.

\medskip

Then there exists a contraction strongly continuous semigroup $e^{-t\Lambda}$ on $C_\infty$ such that
$$
e^{-t\Lambda_n} \rightarrow e^{-t\Lambda} \quad \text{strongly in } C_\infty
$$
locally uniformly in $t \geq 0$.

\end{theorem}

\smallskip

Condition 1) follows from the classical theory, that is, from the fact that $e^{-t\Lambda_n}$ are $L^\infty$ contractions.

Condition 2) is verified as follows. In view of 1), it suffices to verify the existence of the limit on $f$ in a countable dense subset of $C_c^\infty$. Set $u_n:=(\mu+\Lambda_n)^{-1}f$. Fix $R>0$ sufficiently large so that, by Corollary \ref{a1_cor}, $\sup_{\mathbb R^d \setminus B_R(0)}|u|$ is sufficiently small uniformly in $n$. (To this end, we note that $\langle |f|^{p\theta}\rho_x\rangle$, $\langle |f|^{p\theta'}\rho_x\rangle$ in Corollary \ref{a1_cor} are small if $x \in \mathbb R^d \setminus B_R(0)$ for $R$ sufficiently large, i.e.\,$x$ is far away from the support of $f$.) Next, applying Theorem \ref{a1} and the Arzel\`{a}-Ascoli theorem on $\bar{B}_R(0)$, we obtain that there is a subsequence $n_k$ such that $\{u_{{n_k}}\}$ converges uniformly on $\bar{B}_R(0)$. Taking into account the previous observation regarding smallness of $|u_n|$ on $\mathbb R^d \setminus B_R(0)$, we use the diagonal argument to construct a subsequence $u_{{n_\ell}}$  such that the limit
$C_\infty\text{-}\lim_{\ell} (\mu+\Lambda_{{n_\ell}})^{-1}f$ exists. Finally, using the existence of the limit $\text{\small $s\text{-}L^p\text{-}$}\lim_{n} (\mu+\Lambda_
n)^{-1}f$,  $p>\frac{2}{2-\sqrt{\delta}}$, see \cite{KS}, we obtain that the subsequential limit $C_\infty\text{-}\lim_{\ell} (\mu+\Lambda_{{n_\ell}})^{-1}f$ does not depend on the choice of ${n_\ell}$. This gives us condition 2).

Let us verify condition 3). Once again, in view of 1), it suffices to verify 3) on a dense subset of $C_\infty$, e.g.\,all $g \in C_c^\infty$. We invoke the resolvent identity:
\begin{align*}
\mu (\mu+\Lambda_n)^{-1} g - \mu(\mu-\Delta)^{-1} g & = \mu (\mu+\Lambda_n)^{-1} b_n \cdot \nabla (\mu-\Delta)^{-1} g \\
& = (\mu+\Lambda_n)^{-1} b_n \cdot \mu (\mu-\Delta)^{-1} \nabla g.
\end{align*}
Since $ \mu(\mu-\Delta)^{-1} g \rightarrow g$ uniformly as $\mu \rightarrow \infty$, it suffices to show the convergence 
\begin{align}
\label{conv}
\|(\mu+\Lambda_n)^{-1} b_n \cdot \mu (\mu-\Delta)^{-1} \nabla g\|_\infty \leq \|(\mu+\Lambda_n)^{-1} |b_n| \mu (\mu-\Delta)^{-1} |\nabla g| \|_\infty \rightarrow 0 
\end{align}
as $\mu \rightarrow \infty$ uniformly in $n$. This is proved in \cite[Lemma 4]{KS} under additional hypothesis $|b| \in L^2 + L^\infty$, but the proof there can be modified to excludes this hypothesis, see \cite[Lemma 4.16]{KiS_theory}. (Alternatively, one can prove \eqref{conv} using Theorem \ref{a2} below after taking supremum in $x \in \frac{1}{2}\mathbb Z^d$ in \eqref{emb} and using the fact that $f=|\mu (\mu-\Delta)^{-1}g|$ is bounded on $\mathbb R^d$ uniformly in $\mu$.)

\bigskip

\section{Proofs of Theorems \ref{thm2}-\ref{thm5}}

\subsection{Proof of Theorem \ref{thm2}}
(\textit{i}), (\textit{ii}) 
We have 
\[
\langle \partial_t v+\lambda v -\Delta v + b_n\cdot\nabla v, e^v-e^{-v}\rangle=0 \quad \text{ where } v=e^{-\lambda t} u_n. 
\]
Let us introduce the weight function $\zeta_r(x):=\eta\big(\frac{|x|}{r}\big)$, where 
$$
\eta(t):=\left\{
\begin{array}{ll}
1 & \text{ if } \;\; t\leq 1 \\
(1-\theta(t-1)))^\frac{1}{\theta} & \text{ if } \;\; 1<t<1+\theta^{-1}, \quad 0<\theta< \frac{1}{2}\\
0 & \text{ if } \;\; 1+\theta^{-1} \leq t,
\end{array}
\right.
$$
Put $\mathcal C(r,ar)=\{y\in\mathbb R^d\mid r\leq |y|\leq ar\}$, $a=1+\theta^{-1}$. It is easy to check that
\[
|\nabla\zeta_r|\leq r^{-1}\mathbf 1_{\mathcal C(r,ar)} \text{ and } -\Delta\zeta_r \leq(d-1)r^{-2}\mathbf 1_{\mathcal C(r,ar)}.
\]

1.~A direct calculation yields (clearly, $|\nabla\zeta_M|\leq M^{-1}$, $-\Delta v\in L^1$, $|\nabla v|\in L^2\cap L^1$, $v\in L^\infty$):
 \begin{align*}
 \langle-\Delta v,e^v-e^{-v}\rangle&=\lim_{M\to\infty}\langle-\Delta v,\zeta_M(e^v-e^{-v})\rangle\\
 &=\lim_{M\to\infty}\big(\langle|\nabla v|^2,\zeta_M(e^v+e^{-v})\rangle+\langle\nabla v,(e^v-e^{-v})\nabla\zeta_M\rangle\big)\\
 &=\langle|\nabla v|^2,(e^v+e^{-v})\rangle=\langle|\nabla v|^2,(e^\frac{v}{2}-e^{-\frac{v}{2}})^2+2\rangle\\
 &=4\|\nabla(e^\frac{v}{2}+e^{-\frac{v}{2}})\|_2^2+2\|\nabla v\|_2^2.
 \end{align*}
 Therefore,
\[
\lambda\langle v (e^v-e^{-v})\rangle + \partial_t\langle e^v+e^{-v}-2\rangle+2\|\nabla v\|_2^2 + 4\|\nabla (e^\frac{v}{2}+e^{-\frac{v}{2}})\|_2^2+ 2\langle b_n(e^\frac{v}{2}+e^{-\frac{v}{2}}),\nabla (e^\frac{v}{2}+e^{-\frac{v}{2}})\rangle
=0,
\] 
so
\[
\lambda\langle v \sinh v\rangle + \partial_t\langle \cosh v-1\rangle+\|\nabla v\|_2^2+8\|\nabla\cosh \frac{v}{2}\|_2^2\leq 4\|b_n\cosh \frac{v}{2}\|_2\|\nabla\cosh\frac{v}{2}\|_2.
\]
Using our assumption on $b_n$, we write $$\|b_n\cosh\frac{v}{2}\|^2_2= \|b_n\big(\zeta_R\cosh\frac{v}{2}\big)\|^2_2\leq 4\|\nabla \big(\zeta_R\cosh\frac{v}{2}\big)\|^2_2+c(4)\|\zeta_R\cosh\frac{v}{2}\|_2^2$$ with 
$R$ such that $\sprt b_n \subset B_R$ (for this, we increase $R$ slightly, or simply redenote $R+\frac{1}{n}$ from the assumption on $b_n$ by $R$), where, setting $w:=\cosh\frac{v}{2}$, we have 
\begin{align*}
\|\nabla \big(\zeta_R\cosh\frac{v}{2}\big)\|^2 \equiv \|\nabla(\zeta_R w)\|_2^2&=\|\zeta_R\nabla w\|_2^2+\|w\nabla\zeta_R\|_2^2+\langle\zeta_R\nabla\zeta_R,\nabla w^2\rangle\\
&=\|\zeta_R\nabla w\|_2^2-\langle\zeta_R\Delta\zeta_R,w^2\rangle\\
&\leq \|\zeta_R\nabla w\|_2^2+(d-1)R^{-2}\langle\zeta_R w^2\rangle.
\end{align*}
Therefore,
\begin{align*}
4\|b_n \cosh \frac{v}{2}\|_2\|\nabla \cosh \frac{v}{2}\|_2 & \leq \|b_n \cosh \frac{v}{2}\|_2^2 + 4\|\nabla \cosh \frac{v}{2}\|_2^2 \\
& \leq 8\|\nabla \cosh \frac{v}{2}\|_2^2+c_5\langle\zeta_R w^2\rangle, \quad c_5=c(4)+4(d-1)R^{-2} 
\end{align*}
and so
\[
\lambda\langle v\sinh v\rangle+\partial_t\langle\cosh v-1\rangle + \|\nabla v\|_2^2\leq c_5\|\mathbf 1_{B_{aR}}\cosh \frac{v}{2}\|_2^2.
\] 
Since $v\sinh v\geq \cosh v-1=2(\cosh^2\frac{v}{2}-1)$,
\[
 (\lambda-2^{-1}c_5)\langle v\sinh v\rangle+\partial_t\langle\cosh v-1\rangle + \|\nabla v\|_2^2\leq c_5\|\mathbf 1_{B_{aR}}\|_1,
 \]
  or setting $\lambda=2^{-1}c_5$ and changing $v$ to $\frac{v}{c}$, $c>0$,  
\[
\langle\cosh \frac{v(t)}{c}-1\rangle +\int_0^t\|\nabla \frac{v(s)}{c}\|_2^2ds\leq\langle\cosh \frac{f}{c}-1\rangle+tc_5\|\mathbf 1_{B_{aR}}\|_1. \tag{$\star$}
\] 
From $(\star)$ we obtain that $\int_0^t\|\nabla v(s)\|_2^2ds\leq c_1^2(\langle\cosh \frac{f}{c_1}-1\rangle+tc_5\|\mathbf 1_{B_{aR}}\|_1)$ with $c_1=\|f\|_\Phi$.
 Therefore,
\[
\int_0^t\|\nabla v(s)\|_2^2ds\leq (1+tc_5\|\mathbf 1_{B_{aR}}\|_1)\|f\|_\Phi^2.\tag{$\star_1$}
\]
From $(\star)$ we obtain also the inequality
\[
\langle\cosh \frac{v(\frac{t}{m})}{c}-1\rangle\leq \langle\cosh \frac{f}{c}-1\rangle+c_5\|\mathbf 1_{B_{aR}}\|_1\frac{t}{m},\quad m=1,2,\dots \tag{$\star_2$}
\]

\medskip

2.~Set $c=\frac{\|f\|_\Phi}{1-\gamma_m}$, $m\geq m_0$, where $\gamma_m=c_5\|\mathbf 1_{B_{aR}}\|_1\frac{t}{m}$ and $\gamma_{m_0}<1$. Then,
due to 
\[
\langle\cosh \frac{f}{(1-\gamma_m)c}-1\rangle\leq 1 \text{ and inequality } \cosh \frac{f}{c}-1\leq(1-\gamma_m)\big(\cosh \frac{f}{(1-\gamma_m)c}-1\big),
\]
we obtain from $(\star_2)$ that
\[
\big\langle\cosh \frac{v(\frac{t}{m})}{c}-1\big\rangle\leq 1-\gamma_m+\gamma_m=1,\quad \text{i.e. } \big\|v(\frac{t}{m})\big\|_\Phi\leq c=\frac{1}{1-\gamma_m}\|f\|_\Phi.
\]
Therefore, setting $G=c_5\|\mathbf 1_{B_{aR}}\|_1$, and using semigroup property of $v(t)$, we arrive at
\[
\|v(t)\|_\Phi\leq (1-G\frac{t}{m})^{-m}\|f\|_\Phi
\]
and so
\[
\|v(t)\|_\Phi\leq e^{Gt}\|f\|_\Phi.
\]
Thus, setting $T^t_nf:= u_n(t)$,
\begin{equation}
\label{T_n_conv}
 \|T_n^tf\|_\Phi\leq e^{(2^{-1}c_5+G)t}\|f\|_\Phi.
\end{equation}
Thus, every $T_n^t$ admits extension by continuity from $C_c^\infty$ to $L_\Phi$, which we denote again by $T_n^t$.
We have $\lim_{t\downarrow 0}\|T_n^t f-f\|_\Phi=0$ for all $f \in C_c^\infty$. Since $n$ is finite, the latter is evident from the classical theory, which allows to pass to the limit in $n$ under the gague norm of $T_n^tf-f$. Now, combined with \eqref{T_n_conv},
this yields $$s\mbox{-}L_\Phi\mbox{-}\lim_{t\downarrow 0}T^t_n=1, \quad n \geq 1,$$ i.e.\,semigroups $T_n^t$ are strongly continuous. (So, we can write $T_n^t=e^{-t\Lambda_n}$, where generator $\Lambda_n$ should be considered as appropriate operator realization of $-\Delta + b \cdot \nabla$ in $L_\Phi$. For the sake of uniformity, however, we will continue to use notation $T^t_n$ throughout the rest of the proof.)

\medskip

3.~Next, we claim that $\{T^t_nf\}$ is a Cauchy sequence in $L^\infty([0,T],\mathcal L_\Phi)$ and in $L^2([0,T], W^{1,2}(\mathbb R^d))$. Indeed, set $h=\frac{v_n-v_k}{c}$, $c>0$. Then
\[
\lambda h + \partial_t h -\Delta h + b_n\cdot\nabla h = c^{-1}(b_k-b_n)\cdot\nabla v_k, \;\;h(0)=0,
\]
so
\[
\sup_{0\leq s\leq t}\langle\cosh h(s)-1\rangle+\int_0^t\|\nabla h(s)\|_2^2ds\leq c_5\langle\mathbf 1_{B_{aR}}\rangle  t + c^{-1}e^{2c^{-1}\|f\|_\infty}\int_0^t\langle|b_k-b_n||\nabla v_k(s)|\rangle ds. \tag{$\star\star$}
\]
We estimate, using ($\star_1$),
\begin{align*}
\int_0^t\langle|b_k-b_n||\nabla v_k(s)|\rangle ds & \leq\bigg(\int_0^t\|b_k-b_n\|_2^2ds\bigg)^\frac{1}{2}\bigg(\int_0^t\|\nabla v_k(s)\|_2^2ds\bigg)^\frac{1}{2}\\
&\leq \sqrt{t}\|b_k-b_n\|_2(1+tG)^\frac{1}{2}\|f\|_\Phi.
\end{align*}
Thus, for every fixed $c>0$, $t>0$,
\[
\lim_{n,k\to\infty}\sup_{0\leq s\leq t}\langle\cosh h(s)-1\rangle+\lim_{n,k\to\infty}\int_0^t\|\nabla h(s)\|_2^2ds\leq c_5\langle\mathbf 1_{B_{aR}}\rangle t. 
\]
In particular, $\lim_{n,k\to\infty}\int_0^t\|\nabla (v_n(s)-v_k(s))\|_2^2ds\leq c^2 c_5\langle\mathbf 1_{B_{aR}}\rangle t$ for any $c>0$, i.e.
\[
\lim_{n,k\to\infty}\int_0^t\|\nabla v_n(s)-\nabla v_k(s)\|_2^2ds=0.
\] 

Now fix $t_0$ by $c_5\langle\mathbf 1_{B_{aR}}\rangle t_0\leq 1$, then
\[
\lim_{n,k\to\infty}\sup_{0\leq s\leq t_0}\langle\cosh h(s)-1\rangle\leq 1 \text{ for any } c>0.
\]
 The latter means that $\lim_{n,k\to\infty}\sup_{0\leq s\leq t_0}\|v_n(s)-v_k(s)\|_\Phi=0$. The claim is established.

\medskip

4. Set $T^t f:=L_\Phi\mbox{-} \lim_n T^t_n f$, $f \in C_c^\infty$. Then, clearly, by \eqref{T_n_conv}
$$
 \|T^tf\|_\Phi\leq e^{(2^{-1}c_5+G)t}\|f\|_\Phi.
$$
We extend $T^t$ by continuity from $C_c^\infty$ to $L_\Phi$.
Then, clearly, $T^{t+s}=T^tT^s$,
 \[
s\mbox{-}L_\Phi\mbox{-}\lim_{t\downarrow 0}T^t=1.
 \]
This is the sought semigroup $e^{-t\Lambda_\Phi}:=T^t$.
Moreover, in view of ($\star_1$), we have
$$
T^tg\in L^2([0,T], W^{1,2}(\mathbb R^d))\quad g\in L_\Phi.
$$
The weak solution characterization of $u=T^t g$ now follows right away from the convergence results established above. The proof of (\textit{i}), (\textit{ii}) is completed.

\medskip

(\textit{iii}) This uniqueness result follows right away from the construction of the semigroup $T^t$ by verifying Cauchy's criterion.  \hfill \qed

\subsection{Proof of Theorem \ref{thm3}}

Set $v:=e^{-t(\lambda+\Lambda(b_n)}f$. We have $$\lambda\langle v\sinh v\rangle+\partial_t\langle \cosh v-1\rangle+\|\nabla v\|_2^2+ 8\|\nabla \cosh \frac{v}{2}\|_2^2=-4\langle b_n\cosh \frac{v}{2},\nabla \cosh \frac{v}{2}\rangle.$$        
Put $w:=\cosh \frac{v}{2}$. Let us first establish the estimate
\begin{equation}
\label{star}
4|\langle b_nw,\nabla w\rangle|\leq 8\|\nabla w\|_2^2+c_5\langle\zeta_{R_1}w^2\rangle+\|b^{(2)}w\|_2^2. 
\end{equation}
Writing $b_n^{(1)}=(b_{n,1}^{(1)}, b_{n,2}^{(1)},\dots,b_{n,d}^{(1)})$, $b^{(2)}=(b_1^{(2)}, b_2^{(2)},\dots, b_d^{(2)})$ and using the assumptions and inequality $4|\alpha\beta|\leq |\alpha|^2+4\|\beta|^2$, we have
\[
\langle b_{n,i}^{(1)}w,\nabla_iw\rangle=\langle b_{n,i}^{(1)}\zeta_{R_1}w,\mathbf 1_{\sprt b_{n,i}^{(1)}} \nabla_iw\rangle,\quad \langle b_i^{(2)}w,\nabla_iw\rangle=\langle b_i^{(2)}w,\mathbf 1_{\sprt b_i^{(2)}}\nabla_iw\rangle,
\]
so
\begin{align*}
4|\langle b_nw,\nabla w\rangle|&\leq \|b_n^{(1)}\zeta_{R_1}w\|_2^2+\|b^{(2)}w\|_2^2+4\sum_{i=1}^d\langle(\mathbf 1_{\sprt b_{n,i}^{(1)}}+\mathbf 1_{\sprt b_i^{(2)}})|\nabla_i w|^2\rangle\\
&\leq \|b_n^{(1)}\zeta_{R_1}w\|_2^2+\|b^{(2)}w\|_2^2+4\|\nabla w\|_2^2\\
&\leq 4\|\nabla(\zeta_{R_1}w)\|_2^2+c_4\|\zeta_{R_1}w\|_2^2 + \|b^{(2)}w\|_2^2+4\|\nabla w\|_2^2.
\end{align*}
Recalling that $\|\nabla(\zeta_{R_1}w)\|_2^2\leq\|\nabla w\|_2^2+(d-1)R_1^{-2}\langle\zeta_{R_1}w^2\rangle$, we arrive at \eqref{star}.

The proof of the crucial bounds $$\int_0^t\|e^{-\lambda s}\nabla T^s_nf\|_2^2ds\leq (1+tG)\|f\|_\Phi^2, \quad \|T^t_nf\|_\Phi\leq e^{(\lambda+G)t}\|f\|_\Phi$$ follows the proof of Theorem \ref{thm2}, i.e.\,using \eqref{star} we obtain inequality
\[
\lambda\langle v\sinh v\rangle+\partial_t\langle\cosh v-1\rangle+\|\nabla v\|_2^2\leq (c_4+(d-1)2R_1^{-2})\langle\mathbf 1_{B_{aR_1}}\cosh^2\frac{v}{2} \rangle+\langle|b^{(2)}|^2\cosh^2\frac{v}{2}\rangle,
\]
and hence inequality $\partial_t\langle\cosh v-1\rangle+\|\nabla v\|_2^2\leq  G.$ Integrating the latter over $[0,t]$, we have
\[
\langle\cosh v(t)-1\rangle+\int_0^t\|\nabla v(s)\|_2^2ds\leq \langle\cosh f-1\rangle+ G t.
\]
The rest of the proof essentially repeats the proof of Theorem \ref{thm2}. \hfill \qed

\subsection{Proof of Theorem \ref{thm4}}
We start with identity 
\[
\lambda^\prime\langle v \sinh v\rangle + \partial_t\langle \cosh v-1\rangle +\|\nabla v\|_2^2 + 8\|\nabla \cosh\frac{v}{2}\|_2^2=-\langle b_n\cdot\nabla (\cosh v-1)\rangle,\quad v=e^{-t(\lambda^\prime+\Lambda(b_n))}f.
\]
Let us estimate $-\langle b_n\cdot\nabla (\cosh v-1)\rangle$ from above.
 Define $\hat{\eta}(t)$ to be $1$ if $t\leq 1$, $2-t$ if $1<t<2$ and $0$ if $t\geq 2$. Set $\eta_R(x)=\hat{\eta}(\frac{|x|}{R})$. Then
\begin{align*}
-\langle b_n\cdot\nabla (\cosh v-1)\rangle&=-\lim_{R\rightarrow\infty}\langle\eta_R b_n\cdot\nabla (\cosh v-1)\rangle\\
&=\lim_R\langle\eta_R {\rm div}\, b_n,\cosh v-1\rangle+\lim_R\langle\nabla\eta_R,b_n(\cosh v-1)\rangle\\
&\leq \langle E_nV_1,\cosh v-1\rangle+\langle E_nV_2,\cosh v-1\rangle;
\end{align*}
\begin{align*}
 \langle E_nV_1,\cosh v-1\rangle=&2\langle V_1,E_n(\cosh^2\frac{v}{2}-1)\rangle =-2\langle V_1\rangle+2\big\|V_1^\frac{1}{2}(\zeta_{R_1}\sqrt{E_n\cosh^2\frac{v}{2} })\big\|_2^2\\
 &\leq 8\big\|\nabla\big(\zeta_{R_1}\sqrt{E_n\cosh^2\frac{v}{2}}\big)\big\|_2^2+2c_4\big\langle\zeta_{R_1}E_n\cosh^2 \frac{v}{2}\big\rangle
\end{align*}
and setting $w=\cosh\frac{v}{2}$
\begin{align*}
\big\|\nabla(\zeta_{R_1}\sqrt{E_nw^2 })\big\|_2^2&=\big\|\zeta_{R_1}\nabla\sqrt{E_nw^2}\big\|_2^2-\big\langle\zeta_{R_1}\Delta\zeta_{R_1},E_nw^2\big\rangle\\
&\leq \|\zeta_{R_1}\nabla \sqrt{E_nw^2}\|_2^2+(d-1)R_1^{-2}\langle E_n\zeta_{R_1},w^2\rangle\\
\bigg(\text{we are using } &\big|\nabla\sqrt{E_nw^2}\big|=\frac{\big|E_n(w\nabla w)\big|}{\sqrt{E_nw^2}}\leq\sqrt{E_n|\nabla w|^2}\bigg)\\
&\leq \big\|\nabla w\big\|_2^2+(d-1)R_1^{-2}\big\langle E_n\zeta_{R_1},w^2\big\rangle.
\end{align*}
Thus,
 $-\langle b_n\cdot\nabla (\cosh v-1)\rangle\leq 8\|\nabla \cosh\frac{v}{2}\|_2^2+[2c_4+(d-1)8R_1^{-2}]\langle E_n\zeta_{R_1},\cosh^2 \frac{v}{2}\rangle+\langle E_nV_2,\cosh v-1\rangle$ and the inequality
 \[
\lambda^\prime\langle v \sinh v\rangle + \partial_t\langle \cosh v-1\rangle +\|\nabla v\|_2^2\leq [2c_4+(d-1)8R_1^{-2}]\langle E_n\zeta_{R_1},\cosh^2 \frac{v}{2}\rangle+\langle E_nV_2,\cosh v-1\rangle
\]
is derived and yields (with $\lambda^\prime=\lambda+\|V_2\|_\infty$)
\[
\partial_t\langle \cosh v-1\rangle+\|\nabla v\|_2^2\leq G.
\]
The rest of the proof is practically identical to the proof of Theorem \ref{thm2}. \hfill \qed

\begin{remark}
\label{thm4_rem}
As we noted earlier, one can remove condition $|b|\in L^2$ in Theorem \ref{thm4} by considering $\tilde{b}=b+\mathsf{f}$, where $b$ satisfies the assumptions of Theorem \ref{thm4}, and $|\mathsf{f}|\in L^\infty$, ${\rm div\,} \mathsf{f}\in L^1_\loc$, $V_3:=0\vee {\rm div\,} \mathsf{f}\in L^\infty$.
Indeed, set $\tilde{b}_n=b_n+\mathsf{f}$ and let $v=e^{-t(\lambda^\prime+\Lambda(\tilde{b}_n))}f$, where $\lambda^\prime=\lambda+\|V_2\|_\infty +\|V_3\|_\infty$. Then
\[
\partial_t\langle \cosh v-1\rangle+\|\nabla v\|_2^2\leq G,
\]
\[
\int_0^t\langle|\tilde{b}_k-\tilde{b}_n||\nabla v_k(s)|\rangle ds\rightarrow 0 \text{ as } k,n\rightarrow\infty
\]
due to $|\tilde{b}_k-\tilde{b}_n|=|b_k-b_n|$.
\end{remark}

\subsection{Proof of Theorem \ref{thm5}}
Clearly, we are left to estimate $4\|b_nw\|_2\|\nabla w\|_2$, where $b_n=b\mathbf 1_{|b|\leq n}$, $w=\cosh \frac{v}{2}$, and $v=e^{-t\lambda}u_n$, as follows. Set $\varrho_{R_m}(x)=\zeta_{R_m}(x-x_m)$. We have
\[
\|b_nw\|_2^2=\sum_{m=1}^\infty\|b_n^{(m)}\varrho_{R_m}w\|_2^2\leq\sum_{m=1}^\infty\delta_m\|\nabla(\varrho_{R_m}w)\|_2^2+\sum_{m=1}^\infty c(\delta_m)\|\varrho_{R_m}w\|_2^2,
\]
\begin{align*}
\|\nabla(\varrho_{R_m}w)\|_2^2&=\|\varrho_{R_m}\nabla w\|_2^2+\|w\nabla\varrho_{R_m}\|_2^2+\langle\varrho_{R_m}\nabla\varrho_{R_m},\nabla w^2\rangle\\
&=\|\varrho_{R_m}\nabla w\|_2^2-\langle\varrho_{R_m}\Delta\varrho_{R_m},w^2\rangle\\
&\leq \|\nabla w\|_2^2 + (d-1)R_m^{-2}\langle\varrho_{R_m}w^2\rangle,\\
\sum_{m=1}^\infty\delta_m\|\nabla(\varrho_{R_m}w)\|_2^2&\leq 4\|\nabla w\|_2^2+(d-1)\sum_{m=1}^\infty\delta_mR_m^{-2}\langle\varrho_{R_m}w^2\rangle,\\
4\|b_nw\|_2\|\nabla w\|_2&\leq 8\|\nabla w\|_2^2+\sum_{m=1}^\infty C_m\langle\varrho_{R_m}w^2\rangle, \quad C_m=c(\delta_m)+4(d-1)\delta_mR_m^{-2}.
\end{align*}
Finally,
\begin{align*}
\sum_{m=1}^\infty C_m\langle\varrho_{R_m}w^2\rangle&\leq\sum_{m=1}^\infty C_m\langle\mathbf 1_{B(x_m,aR_m)}w^2\rangle\\
&\leq \langle w^2-1\rangle \sum_{m=1}^\infty C_m + \omega_da^d\sum_{m=1}^\infty C_mR^d.
\end{align*} 
\hfill \qed

\bigskip

\appendix

\section{H\"{o}lder continuity of solutions and embedding theorems}

Throughout this section, $b \in \mathbf{F}_\delta$, $\delta<4$. We use notations introduced in the previous sections: $b_n=E_{\varepsilon_n} b$, $\varepsilon_n \downarrow 0$, and $$\Lambda_n:=-\Delta + b_n \cdot \nabla.$$

\begin{theorem}[{\cite[Theorem 5]{Ki_multi}}]
\label{a1}
The classical solution $u=u_n$ to non-homogeneous equation
\begin{equation}
\label{eq22}
\big(\mu+\Lambda_n\big)u=f, \quad f \in C_c^\infty, \quad \mu>0,
\end{equation}
is H\"{o}lder continuous in every ball $B_1(x)$ with constants that do not depend on $n$ (i.e.\,boundedness or smoothness of $b_n$) or $x \in \mathbb R^d$.
\end{theorem}

\begin{theorem}[{special case of \cite[Theorem 6]{Ki_multi}}]
\label{a2}
Let $u=u_n$ denote the classical solution to non-homogeneous equation
\begin{equation}
\label{eq7}
(\mu+\Lambda_n)u=|b_n|f, \quad f \in C \cap L^1.
\end{equation}
Then for fixed $1<\theta<\frac{d}{d-2}$ and $p>\frac{2}{2-\sqrt{\delta}}$, $p \geq 2$, there exist constants  $\mu_1>0$, $\kappa$, $C$ and $\beta \in ]0,1[$ independent of $n$ such that, for every $x \in \mathbb R^d$,
\begin{align}
\sup_{B_\frac{1}{2}(x)}|u| & \leq C \biggl((\mu-\mu_1)^{-\frac{1}{p\theta}} \langle \big(\mathbf{1}_{|b_n|>1} + |b_n|^{p\theta}\mathbf{1}_{|b_n| \leq 1}\big)|f|^{p\theta}\rho_x\rangle^{\frac{1}{p\theta}} \notag \\
& + \mu^{-\frac{\beta}{p}}\langle \big(\mathbf{1}_{|b_n|>1} + |b_n|^{p\theta'}\mathbf{1}_{|b_n| \leq 1}\big)|f|^{p\theta'}\rho_x\rangle^{\frac{1}{p\theta'}} \biggr) \label{emb}
\end{align}
for all $\mu>\mu_1$, where $\rho_x(y):=\rho(y-x)$, $
\rho(y)=(1+\kappa|y|^{2})^{-\frac{d}{2}-1}$, $y \in \mathbb R^d$.

It follows that
\begin{align*}
\|u\|_\infty \leq C  \sup_{x \in \frac{1}{2}\mathbb Z^d}\biggl( & (\mu-\mu_0)^{-\frac{1}{p\theta}}\big\langle \big(\mathbf{1}_{|b_n|>1} + |b_n|^{p\theta}\mathbf{1}_{|b_n| \leq 1}\big) |f|^{p\theta}\rho_x\big\rangle^{\frac{1}{p\theta}} \\
&+ \mu^{-\frac{\beta}{p}}\big\langle \big(\mathbf{1}_{|b_n|>1} + |b_n|^{p\theta'}\mathbf{1}_{|b_n| \leq 1}\big)|f|^{p\theta'}\rho_x\big\rangle^{\frac{1}{p\theta'}} \biggr).
\end{align*}
\end{theorem}

To make the paper self-contained, below we reproduce more or less verbatim the proofs of \cite[Theorem 5]{Ki_multi} and \cite[Theorem 6]{Ki_multi}.

\subsection{Proof of Theorem \ref{a1}}
Fix throughout this proof $p>\frac{2}{2-\sqrt{\delta}}$, $p \geq 2$. 
Set $$v:=(u-k)_+, \quad k \in \mathbb R.$$ Fix $R_0 \leq 1$.

\begin{proposition}[{\cite[Prop.\,1]{Ki_multi}, Remark 12}]
\label{c_prop} 
For all $0<r<R \leq R_0$, 
\begin{align*}
\|(\nabla v^{\frac{p}{2}}) \mathbf{1}_{B_{r}}\|_2^2 \leq \frac{K_1}{(R-r)^{2}} \|v^{\frac{p}{2}} \mathbf{1}_{B_{R}}\|_2^2 + K_2\||f-\mu u|^{\frac{p}{2}} \mathbf{1}_{u>k} \mathbf{1}_{B_{R}}\|_2^2
\end{align*}
for generic constants $K_1$, $K_2$ (in particular, independent of $k$ or $r$, $R$). 
\end{proposition}

\begin{lemma}[{\cite[Lemma 7.1]{G}}]
\label{dg_lemma}
If $\{z_m\}_{m=0}^\infty \subset \mathbb R_+$ is a sequence of positive real numbers  such that
$$
z_{m+1} \leq N C_0^m z^{1+\alpha}_m
$$
for some $C_0>1$, $\alpha>0$, and
$$
z_0 \leq N^{-\frac{1}{\alpha}}C_0^{-\frac{1}{\alpha^2}}.
$$
Then
$
\lim_m z_m=0.
$
\end{lemma}

\begin{lemma}[{\cite[Lemma 7.3]{G}}]
\label{lem2}
Let $\varphi(t)$ be a positive function, and assume that there exists a constant $q$ and a number $0<\tau<1$ such that for every $0<R<R_0$
$$
\varphi(\tau R) \leq \tau^\delta \varphi(R) + BR^\beta
$$
with $0<\beta<\delta$, and
$$
\varphi(t) \leq q\varphi(\tau^k R)
$$
for every $t$ in the interval $(\tau^{k+1}R,\tau^k R)$. Then, for every $0<\rho<R< R_0$, we have
$$
\varphi(\rho) \leq C \biggl(\biggl(\frac{\rho}{R}\biggr)^\beta \varphi(R) + B \rho^\beta \biggr)
$$
with constant $C$ that depends only on $q$, $\tau$, $\delta$ and $\beta$.
\end{lemma}

The proof follows De Giorgi's method as it is presented in \cite[Ch.\,7]{G} with appropriate modifications to account for our somewhat different definition of $L^p$ De Giorgi's classes, i.e.\,functions satisfying the inequality in Proposition \ref{c_prop}.

\begin{proposition}[{\cite[Prop.\,2]{Ki_multi}}]
\label{prop71}
For all $0<r<R \leq R_0$, 
\begin{align*}
\sup_{B_{\frac{R}{2}}} u \leq C_1 \biggl(\frac{1}{|B_R|}\langle u^p\mathbf{1}_{B_R \cap \{u>0\}} \rangle \biggr)^{\frac{1}{p}} \biggl(\frac{|B_R \cap \{u>0\}|}{|B_R|}\biggr)^\frac{\alpha}{p} + C_2 R^{\frac{2}{p}}
\end{align*}
for generic constants $C_1$, $C_2$ that also depend on $\|f-\mu u\|_\infty~(\leq 2\|f\|_\infty)$, where $\alpha>0$ is fixed by $\alpha(\alpha+1)=\frac{2}{d}$.
\end{proposition}
\begin{proof}
Without loss of generality, $R_0=1$. Let $\frac{1}{2}<r<\rho \leq 1$. Fix $\eta \in C_c^\infty$, $\eta=1$ on $B_{r}$, $\eta=0$ on $\mathbb R^d \setminus \bar{B}_{\frac{r+\rho}{2}}$, $|\nabla \eta| \leq \frac{4}{\rho-r}$. Set $\zeta:=\eta v=\eta(u-k)_+$, $k \in \mathbb R$. 
Using H\"{o}lder's inequality and Sobolev's embedding theorem, we obtain
\begin{align*}
\|v^{\frac{p}{2}}\mathbf{1}_{B_r}\|_2^2 & \leq \|\zeta^{\frac{p}{2}}\mathbf{1}_{B_r}\|_2^2  \leq \langle \mathbf{1}_{B_{r} \cap \{u>k\}}\rangle^{\frac{2}{d}} 
 \langle \zeta^{\frac{pd}{d-2}}\mathbf{1}_{B_\frac{r+\rho}{2}}\rangle^{\frac{d-2}{d}} \\
& \leq c_1|B_{r} \cap \{u>k\}|^{\frac{2}{d}}\langle |\nabla \zeta^{\frac{p}{2}}|^2\mathbf{1}_{B_{\frac{r+\rho}{2}}}\rangle \\
& = c_1|B_{r} \cap \{u>k\}|^{\frac{2}{d}}\langle |(\nabla \eta^{\frac{p}{2}})v^{\frac{p}{2}}+\eta^{\frac{p}{2}}\nabla v^{\frac{p}{2}}|^2\mathbf{1}_{B_{\frac{r+\rho}{2}}}\rangle
\end{align*}
Hence
\begin{align*}
\|v^{\frac{p}{2}}\mathbf{1}_{B_r}\|_2^2 \leq c_2|B_{r} \cap \{u>k\}|^{\frac{2}{d}}\biggl(\frac{1}{(\rho-r)^2}\|v^{\frac{p}{2}} \mathbf{1}_{B_\frac{r+\rho}{2}} \|_2^2 + \|(\nabla v^{\frac{p}{2}})\mathbf{1}_{B_{\frac{r+\rho}{2}}}\|_2^2 \biggr).
\end{align*}
On the other hand, Proposition \ref{c_prop} yields:
\begin{align}
\label{v_cor}
\|(\nabla v^{\frac{p}{2}}) \mathbf{1}_{B_{\frac{r+\rho}{2}}}\|_2^2 \leq \frac{K_1}{(\rho-r)^{2}} \|v^{\frac{p}{2}} \mathbf{1}_{B_{\rho}}\|_2^2 + K_2\|f-\mu u\|_\infty^p\, \big|B_\rho\cap\{u>k\}\big|,
\end{align}
so
\begin{align}
\|v^{\frac{p}{2}}\mathbf{1}_{B_r}\|_2^2 & \leq C|B_{r} \cap \{u>k\}|^{\frac{2}{d}}\biggl(\frac{1}{(\rho-r)^2}\|v^{\frac{p}{2}} \mathbf{1}_{B_\rho} \|_2^2 + \|f-\mu u\|_\infty^p\, \big|B_\rho\cap\{u>k\}\big| \biggr) \notag \\
& \leq \frac{C|B_{\rho} \cap \{u>k\}|^{\frac{2}{d}}}{(\rho-r)^2}\|v^{\frac{p}{2}} \mathbf{1}_{B_\rho} \|_2^2 + C\|f-\mu u\|_\infty^p |B_{\rho} \cap \{u>k\}|^{1+\frac{2}{d}}. \label{i_8}
\end{align}
Now, returning from notation $v$ to $(u-k)_+$, we note that 
if $h<k$, then $\|(u-k)^{\frac{p}{2}}\mathbf{1}_{B_\rho \cap \{u>k\}}\|_2 \leq \|(u-h)^{\frac{p}{2}}\mathbf{1}_{B_\rho \cap \{u>h\}}\|_2$ and $\|(u-h)^{\frac{p}{2}}\mathbf{1}_{B_\rho \cap \{u>h\}}\|_2^2 \geq (k-h)^p |B_\rho \cap \{u>k\}|$. Therefore, we obtain from \eqref{i_8}
\begin{align*}
\|(u-k)_+^{\frac{p}{2}}\mathbf{1}_{B_r}\|_2^2 & \leq \frac{C}{(\rho-r)^2}\|(u-h)_+^{\frac{p}{2}}\mathbf{1}_{B_\rho}\|_2^2|B_{\rho} \cap \{u>h\}|^{\frac{2}{d}} \\
& + \frac{C\|f-\mu u\|_\infty^p}{(k-h)^p}\|(u-h)_+^{\frac{p}{2}}\mathbf{1}_{B_\rho}\|_2^2|B_{\rho} \cap \{u>h\}|^{\frac{2}{d}}.
\end{align*}
Multiplying this inequality by $|B_r \cap \{u>k\}|^\alpha~~\big(\leq \frac{1}{(k-h)^{p\alpha}} \|(u-h)_+^{\frac{p}{2}}\mathbf{1}_{B_\rho}\|_2^{2\alpha}\big)$ and using $\alpha^2 + \alpha=\frac{2}{d}$, we obtain
\begin{align*}
& \|(u-k)_+^{\frac{p}{2}}\mathbf{1}_{B_r}\|_2^2|B_r \cap \{u>k\}|^\alpha \\
&\leq C \biggl[\frac{1}{(\rho-r)^2} + \frac{\|f-\mu u\|_\infty^p}{(k-h)^p} \biggr]\frac{1}{(k-h)^{p\alpha}} \bigl(\|(u-h)_+^{\frac{p}{2}}\mathbf{1}_{B_\rho}\|_2^2|B_\rho \cap \{u>h\}|^\alpha \bigr)^{1+\alpha}.
\end{align*}
Now, take $r:=r_{i+1}$, $\rho:=r_i$, where $r_i:=\frac{R}{2}(1+\frac{1}{2^i})$ and $k:=k_{i+1}$, $h:=k_i$, where $k_i:=\xi(1-2^{-i})$, with constant $\xi \geq R^{\frac{2}{p}}$ to be chosen later. Then, setting $$z_i=z(k_i,r_i):=\|(u-k_i)_+^{\frac{p}{2}}\mathbf{1}_{B_{r_i}}\|_2^2|B_{r_i} \cap \{u>k_i\}|^\alpha,$$ we have
$$
z_{i+1} \leq K\bigg[2^{2i} + \frac{2^{pi}R^2}{\xi^p}\bigg]\frac{1}{R^2}\frac{2^{pi\alpha}}{\xi^{p\alpha}}z_i^{1+\alpha}
$$
hence (using  $\xi \geq R^{\frac{2}{p}}$)
$$
z_{i+1} \leq 2^{p(1+\alpha)i} \frac{2K }{R^2}\frac{1}{\xi^{p\alpha}}z_i^{1+\alpha}.
$$
We apply Lemma \ref{dg_lemma}. In the notation of this lemma, $C_0=2^{p(1+\alpha)}$ and $N=\frac{2K }{R^2}\frac{1}{\xi^{p\alpha}}$. We need
$$z_0 \leq N^{-\frac{1}{\alpha}}C_0^{-\frac{1}{\alpha^2}}$$
where, recall, $z_0=\langle u^p \mathbf{1}_{B_R \cap \{u>0\}}\rangle |B_R \cap \{u>0\}|^\alpha$. The latter amounts to requiring $$\xi \geq C_1 R^{-\frac{2}{p\alpha}}z_0^{\frac{1}{p}}.$$
Take
$
\xi := R^{\frac{2}{p}} + C_1 R^{-\frac{2}{p\alpha}}z_0^{\frac{1}{p}}.
$
By Lemma \ref{dg_lemma}, $z(\xi,\frac{R}{2})=0$, i.e.\,$\sup_{\frac{R}{2}}u \leq \xi$. The claimed inequality follows.
\end{proof}

Set $${\rm osc\,}(u,R):=\sup_{y,y' \in B_R}|u(y)-u(y')|.$$

\begin{proposition}[{\cite[Prop.\,3]{Ki_multi}}]
\label{osc_prop}
Fix $k_0$ by $$2k_0=M(2R)+ m(2R):=\sup_{B_{2R}}u + \inf_{B_{2R}}u.$$ Assume that $|B_R \cap \{u>k_0\}| \leq \gamma |B_R|$ for some $\gamma<1$. If 
\begin{equation}
\label{hyp_n}
{\rm osc\,}(u,2R) \geq 2^{n+1}C R^{\frac{2}{p}},
\end{equation}
then, for $k_n:=M(2R)-2^{-n-1}{\rm osc\,}(u,2R)$,
$$
|B_{R} \cap \{u>k_n\}| \leq c n^{-\frac{d}{2(d-1)}}|B_R|.
$$
\end{proposition}
\begin{proof}
1.~For $h \in ]k_0,k[$, set $w:=(u-h)^{\frac{p}{2}}$ if $h<u<k$, set $w:=(k-h)^{\frac{p}{2}}$ if $u \geq k$, and $w:=0$ if $u \leq h.$ Note that $w=0$ in $B_R \setminus (B_R \cap \{u>k_0\})$. The measure of this set is greater than $\gamma |B_R|$, so the Sobolev embedding theorem applies and yields
\begin{align*}
(k-h)^{\frac{p}{2}}|B_R \cap \{u>k\}|^{\frac{d-1}{d}} & \leq c_1 \langle w^{\frac{d}{d-1}}\mathbf{1}_{B_R}\rangle^{\frac{d-1}{d}}  \leq c_2\langle |\nabla w| \mathbf{1}_\Delta \rangle \\
& \leq c_2 |\Delta|^{\frac{1}{2}}\langle |\nabla (u-h)^{\frac{p}{2}}|^2 \mathbf{1}_{B_R \cap \{u>h\}}\rangle^{\frac{1}{2}},
\end{align*}
where $$\Delta:=B_R \cap \{u>h\} \setminus (B_R \cap \{u>k\}).$$ 
Now, it follows from Proposition \ref{c_prop} that
\begin{align*}
\langle |\nabla (u-h)^{\frac{p}{2}}|^2 \mathbf{1}_{B_R \cap \{u>h\}}\rangle & \leq \frac{C_3}{R^2}\langle (u-h)^p \mathbf{1}_{B_{2R} \cap \{u>h\}}\rangle + C_4 |B_{2R} \cap \{u>h\}| \\
& \leq C_3 R^{d-2} (M(2R)-h)^p + C_5 R^d.
\end{align*}
For $h \leq k_n$ we have $M(2R)-h \geq M(2R)-k_n \geq CR^{\frac{2}{p}}$, where we have used \eqref{hyp_n}.
Therefore, summarizing what was written above, we have
$$
(k-h)^{\frac{p}{2}}|B_R \cap \{u>k\}|^{\frac{d-1}{d}} \leq c|\Delta|^{\frac{1}{2}}R^{\frac{d-2}{2}}(M(2R)-h)^\frac{p}{2}.
$$

2.~Select $k=k_i:=M(2R)-2^{-i-1}{\rm osc\,}(u,2R)$, $h=k_{i-1}$. Then
$$
M(2R)-h=2^{-i}{\rm osc\,}(u,2R), \quad |k-h|=2^{-i-1}{\rm osc\,}(u,2R),
$$
so
$$
|B_R \cap \{u>k_n\}|^{\frac{2(d-1)}{d}} \leq |B_R \cap \{u>k_i\}|^{\frac{2(d-1)}{d}} \leq C |\Delta_i| R^{d-2},
$$
where $\Delta_i:=B_R \cap \{u>k_i\} \setminus (B_R \cap \{u>k_{i-1}\})$.
Summing up in $i$ from $1$ to $n$, we obtain 
$$
n |B_R \cap \{u>k_n\}|^{\frac{2(d-1)}{d}} \leq C R^{d-2} |B_R \cap \{u>k_0\}| \leq C' R^{2(d-1)},
$$
and the claimed inequality follows.
\end{proof}

\subsubsection*{Proof of Theorem \ref{a1}, completed} Fix $k_0$ by $2k_0=M(2R)+m(2R)$. Without loss of generality, $|B_R \cap \{u>k_0\}| \leq \frac{1}{2}|B_R|$ (otherwise we replace $u$ by $-u$). Set $k_n:=M(2R)-2^{-n-1}{\rm osc\,}(u,2R)>k_0$. By Proposition \ref{prop71},
\begin{align*}
\sup_{B_\frac{R}{2}}(u-k_n) & \leq  C_1 \bigl(\frac{1}{|B_R|}\langle (u-k_n)^p\mathbf{1}_{B_R \cap \{u>k_n\}} \rangle \bigr)^{\frac{1}{p}} \biggl(\frac{|B_R \cap \{u>k_n\}|}{|B_R|}\biggr)^\frac{\alpha}{p} + C_2 R^{\frac{2}{p}} \\& \leq C_1 \sup_{B_R} (u-k_n) \biggl(\frac{|B_R \cap \{u>k_n\}|}{|B_R|}\biggr)^{\frac{1+\alpha}{p}} + C_2 R^{\frac{2}{p}}
\end{align*}
We now apply Proposition \ref{osc_prop} (with, say, $C=1$). Fix $n$ by 
$$
 c n^{-\frac{d}{2(d-1)}} \leq \bigg(\frac{1}{2C_1}\bigg)^{\frac{p}{1+\alpha}}.
$$ Then, if ${\rm osc\,}(u,2R) \geq 2^{n+1} R^{\frac{2}{p}}$, we obtain from Proposition \ref{osc_prop}
$$
M\left(\frac{R}{2}\right)-k_n \leq \frac{1}{2}(M(2R)-k_n) + C_2R^{\frac{2}{p}}, 
$$
so, 
$$
M\left(\frac{R}{2}\right) \leq M(2R) - \frac{1}{2^{n+1}}{\rm osc\,}(u,2R)+\frac{1}{2}\frac{1}{2^{n+1}}{\rm osc\,}(u,2R) +  C_2R^{\frac{2}{p}},
$$
$$
M\left(\frac{R}{2}\right) - m\left(\frac{R}{2}\right) \leq M(2R) - m(2R) - \frac{1}{2}\frac{1}{2^{n+1}}{\rm osc\,}(u,2R) +  C_2R^{\frac{2}{p}}.
$$
Hence, since ${\rm osc\,}(u,2R)=M(2R)-m(2R)$,
\begin{equation}
\label{osc_1}
{\rm osc\,}\biggl(u,\frac{R}{2}\biggr) \leq \biggl(1-\frac{1}{2^{n+2}} \biggr){\rm osc\,}(u,2R)+C_2R^{\frac{2}{p}}.
\end{equation}
If ${\rm osc\,}(u,2R) \leq 2^{n+1}R^{\frac{2}{p}}$, then, clearly,
\begin{equation}
\label{osc_2}
{\rm osc\,}\left(u,\frac{R}{2}\right) \leq \biggl(1-\frac{1}{2^{n+2}} \biggr){\rm osc\,}(u,2R)+ \frac{1}{2} R^{\frac{2}{p}}.
\end{equation}
This yields the sought H\"{o}lder continuity of $u$ via Lemma \ref{lem2} with $\tau=\frac{1}{4}$, $\delta=\log_\tau(1-2^{-n-2})$ and $0<\beta<\frac{2}{p} \wedge \delta$. (Note that the second inequality in the conditions of Lemma \ref{lem2} holds if $q=1$ and $\varphi$ is non-decreasing, which is our case.)  \hfill \qed

\subsection{Proof of Theorem \ref{a2}}
Recall that $v:=(u-k)_+$, where $u=u_n$ solves
$$
(\mu+\Lambda_n)u=|b_n|f, \quad f \in C \cap L^1.
$$
It suffices to carry out the proof for the case $f \geq 0$. We will need

\begin{proposition}[{\cite[Prop.\,4]{Ki_multi}}]
\label{c_prop2}
 Fix $R_0 \leq 1$ and $p>\frac{2}{2-\sqrt{\delta}}$, $p \geq 2$.
Then, for all $0<r<R \leq R_0$ and every $k \geq 0$,
\begin{align*}
\mu \|v^{\frac{p}{2}}\mathbf{1}_{B_{r}}\|_2^2 + \|(\nabla v^{\frac{p}{2}}) \mathbf{1}_{B_{r}}\|_2^2 & \leq \frac{K_1}{(R-r)^{2}} \|v^{\frac{p}{2}} \mathbf{1}_{B_{R}}\|_2^2 \\
&+ K_2\|\big(\mathbf{1}_{|b_n|>1} + |b_n|^\frac{p}{2}\mathbf{1}_{|b_n| \leq 1}\big)|f|^{\frac{p}{2}} \mathbf{1}_{u>k} \mathbf{1}_{B_{R}}\|_2^2
\end{align*}
for constants $K_1$, $K_2$  independent of $r$, $R$, $k$ and $n$.
\end{proposition}

Recall that we have fixed $1<\theta<\frac{d}{d-2}$.

\begin{proposition} [{\cite[Prop.\,5]{Ki_multi}}]
\label{nonhom_prop}
There exists constants $K$  and $\beta \in ]0,1[$ such that, for all $\mu  \geq 1$, 
\begin{equation}
\label{dg_ineq}
\sup_{B_{\frac{1}{2}}} u_+ \leq K \biggl( \langle u_+^{p\theta}\mathbf{1}_{B_1}\rangle^{\frac{1}{p\theta}} + \mu^{-\frac{\beta}{p}}\big\langle (\mathbf{1}_{|b_n|>1} + |b_n|^p\mathbf{1}_{|b_n| \leq 1})^{\theta'}|f|^{p\theta'}\mathbf{1}_{B_1}\big\rangle^{\frac{1}{p\theta'}} \biggr).  
\end{equation}
\end{proposition}

\begin{proof}
Proposition \ref{c_prop2} yields
\begin{align*}
\mu \|v^{\frac{p}{2}}\mathbf{1}_{B_{r}}\|_2^2 + \|v^{\frac{p}{2}}\|^2_{W^{1,2}(B_{r})} & \leq \tilde{K}_1(R-r)^{-2}\|v\|^{p}_{L^p(B_{R})} \\
&+ K_2\|\Theta^{\frac{1}{p}}f\mathbf{1}_{u>k}\|_{L^p(B_R)}^p, \qquad v=(u-k)_+,\;\;k \geq 0,
\end{align*}
where $\Theta:=\mathbf{1}_{|\mathsf{b_n}|>1} + |\mathsf{b_n}|^p\mathbf{1}_{|\mathsf{b_n}| \leq 1}$ and $\tilde{K}_1$, $K_2$ are generic constants.
By the Sobolev embedding theorem,
$$
\mu \|v\|_{L^p(B_r)}^p + \|v\|^{p}_{L^{\frac{pd}{d-2}}(B_{r})}\leq C_1(R-r)^{-2}\|v\|^{p}_{L^p(B_{R})}+C_2\|\Theta^{\frac{1}{p}}f\mathbf{1}_{u>k}\|_{L^p(B_R)}^p.
$$
Next, we estimate the left-hand side from below using interpolation inequality:
$$
\mu^{\beta} \|v\|_{L^q(B_r)}^p \leq \beta\mu \|v\|_{L^p(B_r)}^p + (1-\beta)\|v\|_{L^{\frac{pd}{d-2}}(B_r)}^p, \quad 0<\beta<1, \quad \frac{1}{q}=\beta\frac{1}{p}+(1-\beta)\frac{d-2}{pd}.
$$

Put $\theta_0:=\frac{q}{p}$, so $1<\theta_0<\frac{d}{d-2}$. We fix $\beta \in ]0,1[$ sufficiently small so that $\theta<\theta_0$. 

\medskip

Hence, taking into account that $q=p\theta_0$,
$$
\mu^{\beta} \|v\|_{L^{p\theta_0}(B_r)}^p \leq \tilde{C}_1(R-r)^{-2}\|v\|^{p}_{L^p(B_{R})}+\tilde{C}_2\|\Theta^{\frac{1}{p}}f\mathbf{1}_{u>k}\|_{L^p(B_R)}^p.
$$
Applying H\"{o}lder's inequality in the RHS, we obtain
\begin{equation}
\label{ineq_h_n}
\mu^{\beta} \|v\|_{L^{p\theta_0}(B_r)}^p \leq \tilde{C}_1(R-r)^{-2}|B_{R}|^{\frac{\theta-1}{\theta}}\|v\|^{p}_{L^{p\theta}(B_{R})}+\tilde{C}_2\|\Theta^{\frac{1}{p}}f\mathbf{1}_{u>k}\|_{L^p(B_R)}^p.
\end{equation}
Set $$R_m:=\frac{1}{2}+\frac{1}{2^{m+1}}, \quad m \geq 0,$$
so $B^m \equiv B_{R_m}$ is a decreasing sequence of balls converging to the ball of radius $\frac{1}{2}$. 
By \eqref{ineq_h_n},
\begin{align}
\mu^{ \beta} \|v\|_{L^{p\theta_0}(B^{m+1})}^p & \leq \hat{C}_1 2^{2m}
\|v\|_{L^{p\theta}(B^{m})}^{p}+\tilde{C}_2  \|\Theta^{\frac{1}{p}}f\mathbf{1}_{u>k}\|_{L^p(B^m)}^p \notag \\
& \leq \hat{C}_1 2^{2m}
\|v\|_{L^{p\theta}(B^{m})}^{p} + \tilde{C}_2 H|B^{m} \cap \{v>0\}|^{\frac{1}{\theta}},
\label{theta_d_d-2_n}
\end{align}
where
$$
H:=\langle \Theta^{\theta'}|f|^{p\theta'}\mathbf{1}_{B^0}\rangle^{\frac{1}{\theta'}} \quad (B^0=B_1, \text{\,i.e.\,ball of radius } 1)
$$

On the other hand, by H\"{o}lder's inequality,
$$
 \|v\|_{L^{p\theta}(B^{m+1})}^{p\theta} \leq  \|v\|_{L^{p\theta_0}(B^{m+1})}^{p\theta} \biggl(|B^{m} \cap \{v>0\}| \biggr)^{1-\frac{\theta}{\theta_0}}.
$$
Applying \eqref{theta_d_d-2_n} in the first multiple in the RHS, we obtain
$$
 \|v\|_{L^{p\theta}(B^{m+1})}^{p\theta} \leq \tilde{C}\mu^{-\beta\theta}\biggl( 2^{2\theta m}\|v\|_{L^{p\theta}(B^{m})}^{p\theta}+ H^{\theta}|B^{m} \cap \{v>0\}| \biggr) \biggl(|B^{m} \cap \{v>0\}| \biggr)^{1-\frac{\theta}{\theta_0}}.
$$
Now, put $v_m:=(u-k_m)_+$ where $k_m:=\xi(1-2^{-m}) \uparrow \xi,$ where constant $\xi>0$ will be chosen later.
Then, using $2^{2\theta m} \leq 2^{p\theta m}$ and dividing by $\xi^{p\theta}$,
\begin{align*}
&\frac{1}{\xi^{p\theta}}  \|v_{m+1}\|_{L^{p\theta}(B^{m+1})}^{p\theta} \\
&\leq \tilde{C}\mu^{-\beta\theta}\biggl( \frac{2^{p\theta m}}{\xi^{p\theta}}\|v_{m+1}\|_{L^{p\theta}(B^{m})}^{p\theta}  + \frac{1}{\xi^{p\theta}} H^{\theta}|B^{m} \cap \{u>k_{m+1}\}|\biggr) \biggl(|B^{m} \cap \{u>k_{m+1}\}| \biggr)^{1-\frac{\theta}{\theta_0}}.
\end{align*}
Using that $\mu \geq 1$, we further obtain
\begin{align*}
&\frac{1}{\xi^{p\theta}}  \|v_{m+1}\|_{L^{p\theta}(B^{m+1})}^{p\theta} \\
&\leq \tilde{C}\biggl( \frac{2^{p\theta m}}{\xi^{p\theta}}\|v_{m+1}\|_{L^{p\theta}(B^{m})}^{p\theta}  + \frac{1}{\xi^{p\theta}} \mu^{-\beta\theta} H^{\theta}|B^{m} \cap \{u>k_{m+1}\}|\biggr) \biggl(|B^{m} \cap \{u>k_{m+1}\}| \biggr)^{1-\frac{\theta}{\theta_0}}.
\end{align*}
From now on, we require that constant $\xi$ satisfies $\xi^p \geq \mu^{-\beta}H$, so
\begin{align}
\label{dg_ineq7}
&\frac{1}{\xi^{p\theta}}  \|v_{m+1}\|_{L^{p\theta}(B^{m+1})}^{p\theta} \\
&\leq \tilde{C}\biggl( \frac{2^{p\theta m}}{\xi^{p\theta}}\|v_{m+1}\|_{L^{p\theta}(B^{m})}^{p\theta}  + |B^{m} \cap \{u>k_{m+1}\}|\biggr) \biggl(|B^{m} \cap \{u>k_{m+1}\}| \biggr)^{1-\frac{\theta}{\theta_0}}. \notag
\end{align}
Now, 
\begin{align*}
|B^{m} \cap \{u>k_{m+1}\}| & = \big|B^{m} \cap \big\{(\frac{u-k_m}{k_{m+1}-k_m})^{2\theta}>1\big\}\big| \\
&\leq (k_{m+1}-k_m)^{-p\theta} \langle v^{p\theta}_m \mathbf{1}_{B^{m}} \rangle  = \xi^{-p\theta}2^{p\theta(m+1)} \|v_m\|_{L^{p\theta}(B^{m})}^{p\theta},
\end{align*}
so using in \eqref{dg_ineq7} $\|v_{m+1}\|_{L^{p\theta}(B^{m})} \leq \|v_{m}\|_{L^{p\theta}(B^{m})}$ and applying the previous inequality, we obtain
$$
\frac{1}{\xi^{p\theta}} \|v_{m+1}\|_{L^{p\theta}(B^{m+1})}^{p\theta} \leq \tilde{C}  2^{p\theta m(2-\frac{\theta}{\theta_0})} \biggl(\frac{1}{\xi^{p\theta}}\|v_m\|_{L^{p\theta}(B^{m})}^{p\theta}\biggr)^{2-\frac{\theta}{\theta_0}}.
$$
Denote
$
z_m:=\frac{1}{\xi^{p\theta}}\|v_m\|_{L^{p\theta}(B^{m})}^{p\theta}.
$
Then
$$
z_{m+1} \leq \tilde{C}\gamma^m z_m^{1+\alpha}, \quad m=0,1,2,\dots, \quad \alpha:=1-\frac{\theta}{\theta_0},\;\; \gamma:=2^{p\theta (2-\frac{\theta}{\theta_0})}
$$
and $z_0 = \frac{1}{\xi^{p\theta}}\langle u_+^{p\theta}\mathbf{1}_{B^0} \rangle \leq \tilde{C}^{-\frac{1}{\alpha}}\gamma^{-\frac{1}{\alpha^2}}$ (recall: $B^0:=B_{R_0} \equiv B_1$)
provided that we fix $c$ by $$\xi^{p\theta}:=\tilde{C}^{\frac{1}{\alpha}}\gamma^{\frac{1}{\alpha^2}}\langle u_+^{p\theta}\mathbf{1}_{B^0}\rangle + \mu^{-\beta\theta}H^\theta.$$
Hence, by Lemma \ref{dg_lemma}, $z_m \rightarrow 0$ as $m \rightarrow \infty$. It follows that
$
\sup_{B_{1/2}}u_+ \leq \xi,
$
and the claimed inequality follows. 
\end{proof}

To end the proof of Theorem \ref{a2}, we need to estimate $\langle u_+^{p\theta}\mathbf{1}_{B_1}\rangle^{1/p\theta}$ in the RHS of \eqref{dg_ineq} in terms of  $f$. We will do it by estimating $\langle u_+^{p\theta}\rho\rangle^{1/p\theta}$
and then applying inequality $\rho \geq c\mathbf{1}_{B_1}$ for appropriate constant $c=c_d$.

\begin{proposition}
\label{rho_prop}
There exist generic constants $C$, $k$ and $\mu_0>0$ such that for all $\mu> \mu_0$,
\begin{align}
(\mu-\mu_0) \langle u^p \rho \rangle + \langle |\nabla u^{\frac{p}{2}}|^2 \rho \rangle \leq  C  \big\langle \big(\mathbf{1}_{|b_n|>1} + |b_n|^p\mathbf{1}_{|b_n| \leq 1}\big)|f|^p\rho \big\rangle. \label{c_ineq_rho} 
\end{align}
\end{proposition}

\begin{proof}The proof is standard, i.e.\,we multiply equation \eqref{eq7} by $u|u|^{p-2}$, integrate and then use apply to the drift term quadratic inequality and the form-boundedness condition.
In the term that contain $\nabla \rho$ we apply inequality $|\nabla \rho| \leq (\frac{d}{2}+1)\sqrt{\kappa} \rho$ with $\kappa$ chosen sufficiently small; since our assumption on $\delta$ is a strict inequality $\delta<4$, this choice of $\kappa$ suffices to get rid of the terms containing $\nabla \rho$. 
 The details can be found e.g.\,in \cite{Ki_multi}.
\end{proof}

\subsubsection*{Proof of Theorem \ref{a2}, completed} By Proposition \ref{nonhom_prop}, for all $\mu \geq 1$,
$$
\sup_{y \in B_{\frac{1}{2}}(x)} |u(y)| \leq K \biggl( \langle |u|^{p\theta}\rho_x\rangle^{\frac{1}{p\theta}} + \mu^{-\frac{\beta}{p}}\big\langle \big(\mathbf{1}_{|b_n|>1} + |b_n|^{p\theta'}\mathbf{1}_{|b_n| \leq 1}\big)|f|^{p\theta'}\rho_x\big\rangle^{\frac{1}{p\theta'}} \biggr),
$$
where $\rho_x(y):=\rho(y-x)$, and constant $C$ is generic.
Applying Proposition \ref{rho_prop} to the first term in the RHS (with $p\theta$ instead of $p$), we obtain for all $\mu \geq \mu_0 \vee 1$
\begin{align*}
\sup_{y \in B_{\frac{1}{2}}(x)} |u(y)| \leq C  \biggl( & (\mu-\mu_0)^{-\frac{1}{p\theta}}\big\langle \big(\mathbf{1}_{|b_n|>1} + |b_n|^{p\theta}\mathbf{1}_{|b_n| \leq 1}\big) |f|^{p\theta}\rho_x\big\rangle^{\frac{1}{p\theta}} \\
&+ \mu^{-\frac{\beta}{p}}\big\langle \big(\mathbf{1}_{|b_n|>1} + |b_n|^{p\theta'}\mathbf{1}_{|b_n| \leq 1}\big)|f|^{p\theta'}\rho_x\big\rangle^{\frac{1}{p\theta'}}
\end{align*}
This ends the proof of Theorem \ref{a2}. \hfill \qed

\bigskip

Following the proof of Theorem \ref{a2}, we obtain

\begin{corollary}
\label{a1_cor}
In the assumptions and the notations of Theorem \ref{a2}, if $u=u_n$ solves on $\mathbb R^d$
$(\mu+\Lambda_n)u=f$,
then, for every $x \in \mathbb R^d$,
\begin{align}
\label{emb2}
\sup_{B_\frac{1}{2}(x)}|u|  \leq K \biggl((\mu-\mu_1)^{-\frac{1}{p\theta}} \langle |f|^{p\theta}\rho_x\rangle^{\frac{1}{p\theta}} \notag  + \mu^{-\frac{\beta}{p}}\langle |f|^{p\theta'}\rho_x\rangle^{\frac{1}{p\theta'}} \biggr).
\end{align}
where $K$ does not depend on $x \in \mathbb R^d$ or $n$. 
\end{corollary}

\end{document}